\documentclass[10pt,oneside,a4paper]{article}
\usepackage{mathrsfs}
\usepackage{mathdots}
\usepackage{amsmath}
\usepackage{amsthm}
\usepackage{bbm}
\usepackage{subfloat}
\usepackage{subfigure}
\usepackage{fancyhdr,graphicx}
\usepackage{amsfonts}
\usepackage{amssymb}
\usepackage{latexsym,bm}
\usepackage{newlfont}
\usepackage{latexsym}
\usepackage{caption2}
\usepackage{algorithmic,algorithm}
\usepackage{multirow}
\usepackage[pagewise]{lineno}
\usepackage{slashbox}
\usepackage{booktabs}
\usepackage{float}
\usepackage{array}
\usepackage{tabularx}
\newtheorem{theorem}{Theorem}
\newtheorem{lemma}[theorem]{\quad Lemma}
\newtheorem{example}{Example}
\newtheorem{remark}{Remark}

\makeatletter
\long\def\@makecaption#1#2{%
 \vskip\abovecaptionskip
  \sbox\@tempboxa{{#1.}\quad #2}%
 \ifdim \wd\@tempboxa >\hsize
    { #1.}\quad #2\par
     \else
  \global \@minipagefalse
   \hb@xt@\hsize{\hfil\box\@tempboxa\hfil}%
   \fi
   \vskip\belowcaptionskip}
\makeatother

\title{ \textbf{Second order WSGD operators II: A new family of difference schemes for space fractional advection diffusion equation
}}

\author{Can Li$^{a}$,   Weihua Deng$^{b,}$\footnote{Corresponding author. E-mail address: dengwh@lzu.edu.cn.}\\
\small{$^a$ Department of Applied Mathematics, School of Sciences,  Xi'an University of Technology,}\\
        \small{ Xi'an, Shaanxi 710054, P.R.China}\\
\small{$^{b}$ School of Mathematics and Statistics, Lanzhou University,
Lanzhou 730000, P.R.China}\\
        \vspace{-0.1cm}
        }

\date{}
\begin{document}
\maketitle \makeatletter
\newcommand{\rmnum}[1]{\romannumeral #1} 　　
\newcommand{\Rmnum}[1]{\expandafter\@slowromancap\romannumeral #1@}
\makeatother \vspace{-1cm}
\begin{abstract}
The second order weighted and shifted Gr\"{u}nwald difference (WSGD) operators are developed in [Tian et al.,
arXiv:1201.5949] to solve space fractional partial differential equations. Along this direction, we further design a new family of second order WSGD operators; by properly choosing the weighted parameters, they can be effectively used to discretize space (Riemann-Liouville) fractional derivatives. Based on the new second order WSGD operators, we derive a family of difference schemes for the space fractional advection diffusion equation. By von Neumann stability analysis, it is proved that the obtained schemes are unconditionally stable. Finally, extensive numerical experiments are performed to demonstrate the performance of the schemes and confirm the convergent orders.

\end{abstract}



\section{Introduction}
\setlength{\unitlength}{1cm}
As an extension of classical calculus, fractional calculus also has more than three centuries of history.
However, it seems that the applications of fractional calculus to physics and engineering attract enough
attention only in the last few decades. Nowadays, it has been found that a
broad range of non-classical phenomena in the applied sciences and engineering
can be well described by some fractional kinetic equations \cite{Metzler:00}.
The fractional calculus is becoming more and more popular, especially in describing the anomalous diffusions
\cite{Benson:00,Podlubny:99,Metzler:00,Metzler:04}, which arise in physics, chemistry, biology, and other complex dynamics.
With the appearance of various kinds of fractional partial differential equations (PDEs), finding the ways to effectively solve them becomes a natural topic. Based on the integral transformations, sometimes we can find the analytical solutions of linear fractional PDEs with constant coefficients \cite{Oldham:74,Podlubny:99,Metzler:00}. Even in these cases, most of the time their solutions are expressed by transcendental functions or infinite series. So looking for the numerical solutions of fractional PDEs becomes a realistic expectation in practical applications; and designing efficient and robust numerical schemes for fractional PDEs is the basis to this task.

Over the past decade, the finite difference methods are implemented in simulating the space fractional advection
diffusion equations \cite{Lynch:03,Deng:04,Meerschaert:04,Liu:04,Meerschaert:06,Sousa:09}. However, most of the time,
the accuracy of  finite difference scheme or finite volume method is first order.
In recent years, high order discretizations and fast methods for space
fractional PDEs with Riemann-Liouville fractional derivatives attract many authors' interests.
Based on the Toeplitz-like structure of the difference matrix, Wang et al \cite{wang:10} solve the algebraic equation corresponding to fractional diffusion equation with a $N\log ^2N$ cost; Pang and Sun \cite{Pang:12}
propose a multigrid method to solve the discretized system of the fractional diffusion equation.
By using the linear spline approximation, Sousa and Li provide a second
order discretization for the Riemann-Liouville fractional derivatives and establish an
unconditionally stable finite difference method for one-dimensional
fractional diffusion equation in \cite{Sousaa:11}; the results on two-dimensional
two-sided space fractional convection diffusion equation in finite domain can be seen in
\cite{Chen:13}. Ortigueira \cite{Ortigueira:06} gives the ``fractional centred derivative"
to approximate the Riesz fractional derivative with second order accuracy, and this
scheme is used by \c{C}elik and Duman in \cite{Celik:11} to approximate fractional
diffusion equation with the Riesz fractional derivative in a finite domain.

More recently, Tian et al \cite{Tian:11} propose the second order difference approximations, called weighted and shifted Gr\"{u}nwald difference (WSGD) approximations, to the Riemann-Liouville space fractional derivatives. The basic idea of the WSGD approximation is to cancel the low order terms by combining the Gr\"{u}nwald difference operators with different shifts and weights. Along this direction, this paper further introduces a new  family of WSGD opertors by providing more to be chosen parameters, called second order WSGD operators II. As specific applications, the second order WSGD operators II are used to discretize the space fractional derivative of the space fractional advection diffusion equation. And the von Neumann stability analyses are performed to the obtained numerical schemes. We theoretically prove and numerically verify that with proper chosen parameters for the WSGD space discretizations, the obtained implicit and Crank-Nicolson schemes are both unconditionally von Neumann stable.


The outline of this paper is organized as follows. In Section \ref{sec:2}, we first introduce
a new family of second order discretizations of the Riemann-Liouville fractional derivatives.
Then two kinds of difference schemes for the one dimensional space fractional advection diffusion
equation are presented in Section \ref{sec:3}, and the corresponding numerical stabilities are performed. Furthermore, in
Section \ref{two-dimensions}, the numerical schemes and the proof of numerical stability are extended to two dimensional cases.  Finally, in Section \ref{Numericalsec}, extensive numerical experiments are carried out to demonstrate the performance of the proposed numerical schemes and confirm theoretical analysis.

\section{Second order WSGD discretizations for the Riemann-Liouville fractional derivatives}\label{sec:2}
There are some different definitions for the fractional derivatives, such
as the Gr\"{u}nwald-Letnikov derivative, the Riemann-Liouville
derivative, the Caputo derivative, and other modified definitions for the practical applications \cite{Podlubny:99}. Most of the time, they are not completely equivalent; in particular, the Gr\"{u}nwald-Letnikov and Riemann-Liouville
derivatives are equivalent if ignoring the regularity requirements of the functions being performed. Usually, the more popular space fractional derivative is the Riemann-Liouville derivative.
Let the function $u(x)$ be defined on the finite interval $(a,b)$, then the left and right Riemann-Liouville
fractional derivatives of order $\alpha \,(n-1<\alpha<n)$ are, respectively,  defined by
\begin{equation}\label{1.2} {_{a}}D_x^{\alpha}
u(x)=\frac{1}{\Gamma(n-\alpha)}\frac{\partial^n}{\partial x^n}\int^x_{a}(x-s)^{n-\alpha-1}u(s)ds,
\end{equation}
and
\begin{equation}\label{1.3}
{_x}D_{b}^{\alpha}
u(x)=\frac{(-1)^n}{\Gamma(n-\alpha)}\frac{\partial^n}{\partial x^n}\int^{b}_x(s-x)^{n-\alpha-1}u(s)ds.
\end{equation}
Based on the definition of the Gr\"{u}nwald-Letnikov
derivative and its relationship with the left and right Riemann-Liouville
derivatives, we have  \cite{
Podlubny:99}
\begin{equation}\label{2.3}
_aD_x^\alpha u(x) =\frac{1}{h^\alpha}\sum^{[\frac{x-a}{h}]}_{k=0}w^{(\alpha)}_{k}
u(x-kh)+O(h),
\end{equation}
\begin{equation}\label{2.4}
_xD_b^\alpha u(x)=\frac{1}{h
^\alpha}\sum^{[\frac{b-x}{h}]}_{k=0}w^{(\alpha)}_k u(x+kh)+O(h),
\end{equation}
where  $h=(b-a)/N$,  $N$ is a
positive integer, $\Gamma(\cdot)$ is the gamma function, and $
w^{(\alpha)}_{k}:=\frac{\Gamma(k-\alpha)}{\Gamma(-\alpha)\Gamma(k+1)}=(-1)^k\binom{(\alpha)}{k}
$ is the normalized Gr\"{u}nwald weights. However, the difference scheme based on
(\ref{2.3}) and/or (\ref{2.4}) to discretize space fractional derivatives for the time dependent problems is unconditionally unstable
\cite{Meerschaert:04}. To overcome this problem, Meerschaert and Tadjeran
 firstly introduce the shifted Gr\"{u}nwald-Letnikov approximation formulas \cite{Meerschaert:04}
\begin{equation}\label{2.5}
_aD_x^\alpha u(x)=\frac{1}{h^\alpha}\sum^{[\frac{x-a}{h}]+p}_{k=0}w^{(\alpha)}_{k}
u(x-(k-p)h)+O(h),
\end{equation}
\begin{equation}\label{2.6}
_xD_b^\alpha u(x)=\frac{1}{h ^\alpha}\sum^{[\frac{b-x}{h}]+q}_{k=0}w^{(\alpha)}_{k}
u(x+(k-p)h)+O(h).
\end{equation}
If a function $u(x)$ has support on the finite interval $(a, b)$, then it can be extended to the whole real line $\mathbb{R}$.
In this case the symbol $a$ (or $b$)  in left (right) Riemann-Liouville fractional derivative (\ref{1.2}) ((\ref{1.3}))
can be replaced by $-\infty$ (or $\infty$).
And the following properties hold.
\begin{lemma}\label{lem1}  \cite{Meerschaert:04}
Let $u(x)\in L^1(\mathbb{R})$, ${_{-
  \infty}}D_x^{\alpha+1}u$ and its Fourier transform belong to $L^1(\mathbb{R})$. Let
$p\in\mathbb{R},h>0$ and $1<\alpha\leq2$. Define
\begin{equation*}\label{lem01}
A^{\alpha}_{h,p}u(x):=\frac{1}{h^{\alpha}}\sum_{k=0}^{\infty}w^{(\alpha)}_{k}u(x-(k
-p)h),
\end{equation*}
\begin{equation*}\label{lem20}
B^{\alpha}_{h,p}u(x):=\frac{1}{h^{\alpha}}\sum_{k=0}^{\infty}w^{(\alpha)}_{k}u(x+(k
-p))h).
\end{equation*}
Then
\begin{equation*}\label{lem02}
A^{\alpha}_{h,p}u(x)={_{-\infty}}D_x^{\alpha}u(x)+O(h),
\end{equation*}
\begin{equation*}\label{lem201}
B^{\alpha}_{h,q}u(x)={_x}D_{\infty}^{\alpha}u(x)+O(h),
\end{equation*}
uniformly in $\mathbb{R}$ as $h\rightarrow0$.
\end{lemma}

Inspired by the shifted Gr\"{u}nwald difference operators and working along the direction of the ideas given in \cite{Tian:11}, we derive the following second order approximation for the Riemann-Liouville
fractional derivatives.
\begin{theorem}\label{thm1}
  Let $u(x)\in L^1(\mathbb{R})$, ${_{-
  \infty}}D_x^{\alpha+2}u$ and its Fourier transform belong to $L^1(\mathbb{R})$, and define the left WSGD operator by
  \begin{equation}\label{rWSGD}
  _{_L}\mathcal{D}_{h,p_{1},p_{2},\ldots,p_{m}}^{\alpha,\lambda_{1},\lambda_{2},\ldots,\lambda_{m}}u(x)=\sum_{j=1}^{m}\lambda_{j}A_{h,p_{j}}^{\alpha}u(x),
  \end{equation}
   then, for an integer $m\geq2$, we have
  \begin{equation}\label{thm01}
 _{_L}\mathcal{D}_{h,p_{1},p_{2},\ldots,p_{m}}^{\alpha,\lambda_{1},\lambda_{2},\ldots,\lambda_{m}}u(x)={_{-\infty}}D_x^{\alpha}u(x)+O(h^2),
  \end{equation}
  uniformly for $x\in\mathbb{R}$, where $p_{j},\lambda_{j}$ are real numbers and satisfy the following linear system
  \begin{equation}\label{system}
\begin{cases}
\displaystyle
\sum_{j=1}^{m}\lambda_{j}=1,\\
\displaystyle
\sum_{j=1}^{m}\lambda_{j}(p_{j}-\frac{\alpha}{2})=0.
\end{cases}
\end{equation}
\end{theorem}
\begin{proof}
Using the Fourier transform of the left  Riemann-Liouville fractional derivative \cite{Podlubny:99}
\begin{equation}\label{Fourierff}
\mathscr{F}[{_{-\infty}}D_x^{\alpha}u](\omega)=(i\omega)^{\alpha}\mathscr{F}[u](\omega),
\end{equation}
and taking Fourier transform on the left hand of formula (\ref{rWSGD}), we obtain
  \begin{equation}
    \begin{aligned}\label{eq:2.12}
      \mathscr{F}[_{_L}\mathcal{D}_{h,p_{1},p_{2},\ldots,p_{m}}^{\alpha,\lambda_{1},\lambda_{2},\ldots,\lambda_{m}}u(x)](\omega)& =\sum_{j=1}^{m}\lambda_{j}\bigg(\frac{1}{h^\alpha}
      \sum_{k=0}^{\infty}w_k^{(\alpha)}\mathrm{e}^{-(k-p_j)h(i\omega)}\hat{u}(\omega)\bigg) \\
      & = (i\omega)^\alpha\sum_{j=1}^m \lambda_{j}P_{h,j}(i\omega)\hat{u}(\omega),
    \end{aligned}
  \end{equation}
  where
  \begin{equation}\label{lem04}
    P_{h,j}(z)=\mathrm{e}^{p_{j}hz}\bigg(\frac{1-\mathrm{e}^{-hz}}{hz}\bigg)^\alpha=1+(p_{j}-\frac{\alpha}{2})hz+O(|z|^2h^2),~
    z=i\omega,~i=\sqrt{-1}.
  \end{equation}
  Denoting $\mathscr{F}[_{_L}\mathcal{D}_{h,p_{1},p_{2},\ldots,p_{m}}^{\alpha,\lambda_{1},\lambda_{2},\ldots,\lambda_{m}}u(x)](\omega)=\mathscr{F}[{_{-\infty}}D_x^{\alpha}u](\omega)+\hat{\phi}(k,h)$, in view of (\ref{lem04}) and (\ref{system}), we have
  \begin{equation}
    |\hat{\phi}(\omega,h)|\le Ch^2|i\omega|^{\alpha+2}|\hat{u}(\omega)|.
  \end{equation}
  Due to $\mathscr{F}[{_{-\infty}}D_x^{\alpha+2}u](\omega)\in L^1(\mathbb{R})$, we have
  \begin{equation}
    |_{_L}\mathcal{D}_{h,p_{1},p_{2},\ldots,p_{m}}^{\alpha,\lambda_{1},\lambda_{2},\ldots,\lambda_{m}}u-{_{-\infty}}D_x^{\alpha}u|
    =|\phi|\le
    \frac{1}{2\pi}\int_{\mathbb{R}}|\hat{\phi}(\omega,h)|
    \le C\|\mathscr{F}[{_{-\infty}}D_x^{\alpha+2}u](\omega)\|_{L^1}h^2=O(h^2).
  \end{equation}
\end{proof}
With the Fourier transform of the right Riemann-Liouville fractional derivative \cite{Podlubny:99}
$$\mathscr{F}[{_{x}}D_{\infty}^{\alpha}u](\omega)=(-i\omega)^{\alpha}\mathscr{F}[u](\omega),$$
by the similar arguments to the above, we get
\begin{theorem}\label{thm1}
  Let $u(x)\in L^1(\mathbb{R})$, ${_{-
  \infty}}D_x^{\alpha+2}u$ and its Fourier transform belong to $L^1(\mathbb{R})$, and define the right WSGD operator by
   \begin{equation}\label{lWSGD}
   {_R}\mathcal{D}_{h,q_{1},q_{2},\ldots,q_{m}}^{\alpha,\gamma_{1},\gamma_{2},\ldots,\gamma_{m}}u(x)=\sum_{j=1}^{m}\gamma_{j}B_{h,q_{j}}^{\alpha}u(x),
   \end{equation}
   then, for an integer $m\geq2$, we have
  \begin{equation}\label{thm002}
    {_R}\mathcal{D}_{h,q_{1},q_{2},\ldots,q_{m}}^{\alpha,\gamma_{1},\gamma_{2},\ldots,\gamma_{m}}u(x)={_{x}}D_{\infty}^{\alpha}u(x)+O(h^2),
  \end{equation}
  uniformly for $x\in\mathbb{R}$, where $q_{j},\gamma_{j}$ are real numbers and satisfy
  \begin{equation}\label{system0}
\begin{cases}
\displaystyle
\sum_{j=1}^{m}\gamma_{j}=1,\\
\displaystyle
\sum_{j=1}^{m}\gamma_{j}(q_{j}-\frac{\alpha}{2})=0.
\end{cases}
\end{equation}
Here $\gamma_{j}$ are determined by solving the above linear system (\ref{system0}).
\end{theorem}
\begin{remark}\label{re1}
Note that the desired second order accuracy is obtained by weighting the series expansion of the trigonometric polynomials $P_{h,j}(z)$ in frequency space, i.e.,
\begin{equation}\label{pj}
    P_{h,j}(z)=\mathrm{e}^{p_{j}hz}\bigg(\frac{1-\mathrm{e}^{-hz}}{hz}\bigg)^\alpha=
    1+(p_{j}-\frac{\alpha}{2})hz+\left(\frac{p^2_{j}}{2}-\frac{\alpha}{2}p_{j}+\frac{3\alpha^2+\alpha}{24}\right)h^2z^2+O(|z|^3h^3),~
    z=ik.
  \end{equation}
  For example,
\begin{equation}\label{p10}
     \sum_{j=1}^{3}\lambda_{j}P_{h,j}(z)=
    \sum_{j=1}^{3}\lambda_{j}+\sum_{j=1}^{3}\big(\lambda_{j}(p_{j}-\frac{\alpha}{2})\big)hz+C_{\alpha}h^2z^2+O(|z|^3h^3).
  \end{equation}
  where
  \begin{equation}\label{Mlexp2}
C_{\alpha}=
\begin{cases}
\displaystyle
\frac{\alpha-\alpha^2}{4}+\frac{3\alpha^2+\alpha}{24},~(\lambda_{1},\lambda_{2},\lambda_{3})
=(\frac{\alpha}{2},\frac{2-\alpha}{2},0),\\
\displaystyle
\frac{2-\alpha^2}{4}+\frac{3\alpha^2+\alpha}{24},(\lambda_{1},\lambda_{2},\lambda_{3})
=(\frac{2+\alpha}{4},0,\frac{2-\alpha}{4}),\\
\displaystyle
\frac{4-\alpha^2}{4}+\frac{3\alpha^2+\alpha}{24},~(\lambda_{1},\lambda_{2},\lambda_{3})
=(\frac{2+\alpha}{2},\frac{-2-\alpha}{2},1).
\end{cases}
\end{equation}
For the similar method for the classical derivatives one can refer to \cite{Fornberg:98}. And following this
idea we can get higher order accurate difference approximations for the Riemann-Liouville fractional derivatives.
In the following sections, we focus  on
the  second order difference approximations with three parameters $(\lambda_{1},\lambda_{2},\lambda_{3})$.
\end{remark}
Supposing that a function $u(x)$ has support on the finite domain $[a,b]$, then we have
\begin{equation}
    \begin{aligned}\label{eq2.7ww}
      & _aD_x^{\alpha}u(x)=~\sum_{j=1}^{m}
      \bigg[\frac{\lambda_j}{h^\alpha}\sum_{k=0}^{[\frac{x-a}{h}]+p_{j}}w_k^{(\alpha)}u(x-(k-p_{j})h)
\bigg]
      +O(h^2), \\
      & _xD_b^{\alpha}u(x)=~\sum_{j=1}^{m}
      \bigg[\frac{\gamma_j}{h^\alpha}\sum_{k=0}^{[\frac{b-x}{h}]
      +q_{j}}w_k^{(\alpha)}u(x+(k-q_{j})h)\bigg]+O(h^2).
    \end{aligned}
  \end{equation}
If the function $u(x)$ is non-periodic we are more interested in that $p_{j},q_{j}$ are integers
and satisfy $|p_{j}|\leq1,|q_{j}|\leq1,j=1,\ldots,m,$ for numerical consideration.
For the convenience, we usually take $p_{j}=q_{j},\lambda_{j}=\gamma_{j}$ in practical computations.
Here $\lambda_{j}$ are determined by solving the above system of linear equations (\ref{system}).
When $m=2$, the linear system (\ref{system}) has the unique solution
$\lambda_1=\frac{\alpha-2p_{2}}{2(p_{1}-p_{2})},~\lambda_2=\frac{2p_1-\alpha}{2(p_{1}-p_{2})}$ \cite{Tian:11}.
However, for $m\geq3$, using the knowledge of linear algebra, we know that the system (\ref{system})
has infinitely many solutions.
Now we assume that $\lambda_{1}$ is given in linear system of equations (\ref{system}) with $m=3$, then we have
 \begin{equation}\label{1parameters}
\lambda_2=\frac{2(p_{1}-p_{3})\lambda_1+2p_{3}-\alpha}{2(p_{3}-p_{2})},
~\lambda_3=\frac{2(p_{2}-p_{1})\lambda_1-2p_{2}+\alpha}{2(p_{3}-p_{2})},~p_{2}\neq p_{3}.
\end{equation}
If $\lambda_{2}$ is given in linear algebra system (\ref{system}), then we have
 \begin{equation}\label{2parameters}
\lambda_1=\frac{2(p_{2}-p_{3})\lambda_2+2p_{3}-\alpha}{2(p_{3}-p_{1})},
~\lambda_3=\frac{2(p_{1}-p_{2})\lambda_2+2p_{1}+\alpha}{2(p_{3}-p_{1})},~p_{1}\neq p_{3}.
\end{equation}
And if $\lambda_{3}$ is known, then the solution of the linear system (\ref{system}) gives
 \begin{equation}\label{3parameters}
\lambda_1=\frac{2(p_{3}-p_{2})\lambda_3-2p_{2}+\alpha}{2(p_{1}-p_{2})},~\lambda_2=\frac{2(p_{3}-p_{2})\lambda_3+2p_{2}-\alpha}{2(p_{1}-p_{2})},~p_{1}\neq p_{2}.
\end{equation}

Furthermore, if the integers $p_{j}$ satisfy $|p_{j}|\leq1,j=1,\ldots,m$, then the
 $(\lambda_{1},\lambda_{2},\lambda_{3})$-WSGD operators in finite domain $[a,b]$ can be written as the following form
 \begin{equation}
    \begin{aligned}\label{eq2.7ww}
      & _{_L}\mathcal{D}_{h,1,0,-1}^{\alpha,\lambda_{1},\lambda_{2},\lambda_{3}}u(x)=
      \frac{\lambda_1}{h^\alpha}\sum_{k=0}^{[\frac{x-a}{h}]+1}w_k^{(\alpha)}u(x-(k-1)h)
      +\frac{\lambda_2}{h^\alpha}\sum_{k=0}^{[\frac{x-a}{h}]}w_k^{(\alpha)}u(x-kh)
      +\frac{\lambda_3}{h^\alpha}\sum_{k=0}^{[\frac{x-a}{h}]-1}w_k^{(\alpha)}u(x-(k+1)h), \\
      & _{_R}\mathcal{D}_{h,1,0,-1}^{\alpha,\lambda_{1},\lambda_{2},\lambda_{3}}u(x)=
      \frac{\lambda_1}{h^\alpha}\sum_{k=0}^{[\frac{b-x}{h}]
      +1}w_k^{(\alpha)}u(x+(k-1)h)+\frac{\lambda_2}{h^\alpha}\sum_{k=0}^{[\frac{b-x}{h}]}w_k^{(\alpha)}u(x+kh)
      +\frac{\lambda_1}{h^\alpha}\sum_{k=0}^{[\frac{b-x}{h}]
      -1}w_k^{(\alpha)}u(x+(k+1)h),
    \end{aligned}
  \end{equation}
 where the parameters $\lambda_{j},j=1,2,3,$ satisfy the following linear system
 \begin{equation}\label{lambdaparameter}
\begin{cases}
\displaystyle
\lambda_{1}+\lambda_{2}+\lambda_{3}=1,\\
\displaystyle
\lambda_{1}-\lambda_{3}=\frac{\alpha}{2}.
\end{cases}
\end{equation}
With the help of the knowledge
of linear algebra, the solutions of the system of the linear algebraic equations (\ref{lambdaparameter}) can be collected by
the following three sets
 \begin{equation}\label{12parameters}
\mathcal{S}_{1}=\Big\{\lambda_2=\frac{2+\alpha}{2}-2\lambda_{1},
~\lambda_3=\lambda_1-\frac{\alpha}{2},~\lambda_1~ \textrm{is~given}\Big\},
\end{equation}
  \begin{equation}\label{02parameters}
\mathcal{S}_{2}=\Big\{\lambda_1=\frac{2+\alpha}{4}-\frac{\lambda_2}{2},
~\lambda_3=\frac{2-\alpha}{4}-\frac{\lambda_2}{2},~\lambda_2~ \textrm{is~given}\Big\},
\end{equation}
and
\begin{equation}\label{2parameters}
\mathcal{S}_{3}=\Big\{\lambda_1=\frac{\alpha}{2}+\lambda_3,
~\lambda_2=\frac{2-\alpha}{2}-2\lambda_3,~\lambda_3~ \textrm{is~given}\Big\}.
\end{equation}
It produces a second order approximation if we take any choice of $\lambda_j,j=1,2,3$ in $\mathcal{S}_{j},j=1,2,3$.
Particularly, if we take $\lambda_j=0$ for $j=1,2,3$ in $\mathcal{S}_{j},j=1,2,3$, it recovers the second order approximations presented in \cite{Tian:11}.
After rearranging the weights $w_k^{(\alpha)}$,  the second order approximations for the Riemann-Liouville fractional derivatives at point $x_{j}$ give
\begin{equation}
  \begin{aligned}\label{eq:2.7}
    & _aD_x^{\alpha}u(x_{j})=\frac{1}{h^\alpha}\sum_{k=0}^{j+1}g_k^{(\alpha)}u(x_{j-k+1})+O(h^2), \\
    & _xD_b^{\alpha}u(x_{j})=\frac{1}{h^\alpha}\sum_{k=0}^{N-j+1}g_k^{(\alpha)}u(x_{j+k-1})+O(h^2),
  \end{aligned}
\end{equation}
where the weights are
\begin{equation}\label{relation}
\begin{split}
\displaystyle
g_0^{(\alpha)}=\lambda_{1}w_0^{(\alpha)},~g_1^{(\alpha)}=\lambda_{1}w_1^{(\alpha)}+\lambda_{2}w_0^{(\alpha)},\\
\displaystyle
g_k^{(\alpha)}=\lambda_{1}w_k^{(\alpha)}+\lambda_{2}w_{k-1}^{(\alpha)}
    +\lambda_{3}w_{k-2}^{(\alpha)},~k\ge2.
\end{split}
\end{equation}

\section{Second order WSGD schemes: one dimensional case}\label{sec:3}
We first consider the initial boundary value problem of the following one dimensional advection diffusion equation
\begin{equation}\label{Problem}
  \begin{cases}
    u_{t}(x,t)+v_x \partial_{x}u(x,t)=~d_x({_{a}}D_x^{\alpha}+~{_x}D_{b}^{\alpha})u(x,t)+f(x,t),
     &   \text{$(x,  t) \in (a,b)\times [0, T]$,} \\
    u(x, 0)=u_0(x),    & \text{$x\in (a,b)$},\\
    u(a, t)=\varphi_{1}(t), u(b, t)=\varphi_{2}(t),&\text{$t\in[0, T]$},
  \end{cases}
\end{equation}
where $u(x,t)$ can stand for the concentration of a solute at a point $x$ at
time $t$ \cite{Benson:00}, $f(x,t)$ is the source term, $v_x(>0)$ and $d_x(>0)$ are, respectively, the advection and diffusion coefficients, and ${_{a}}D_x^{\alpha}$ and ${_x}D_{b}^{\alpha}$ are the left and right Riemann-Liouville
fractional derivatives of order $\alpha$, respectively, defined by (\ref{1.2}) and (\ref{1.3}) with $1<\alpha<2$.

For the numerical approximations, we define
$t_n=n\tau,n=0,1,\cdots,N_{t}$, $h=\frac{b-a}{N_{x}}$ is the equidistant grid size in
space, $x_j=a+jh$ for $j=0,1,\cdots,N_{x}$, so that $a\leq x_j\leq b$.
Denote  $u_j^n=u(x_j, t_n), \ \ f_j^{n}=f(x_j, t_{n}).$
In the space discretizations, we choose the WSGD operators ${_L}\mathcal{D}_{h,1,0,-1}^{\alpha,\lambda_{1},\lambda_{2},\lambda_{3}}u(\cdot,t)$ and ${_R}\mathcal{D}_{h,1,0,-1}^{\alpha,\lambda_{1},\lambda_{2},\lambda_{3}}u(\cdot,t)$ to approximate the Riemann-Liouville fractional derivatives ${_a}D_{x}^{\alpha}u(\cdot,t)$ and ${_x}D_{b}^{\alpha}u(\cdot,t)$, respectively. And the first order space derivative is approximated by the standard central difference. If the time derivative is approximated by the implicit Euler discretization, we have
\begin{equation*}
    \frac{u_j^{n+1}-u_j^n}{\tau}+v_x\frac{u_{j+1}^{n+1}-u_{j-1}^{n+1}}{2h}
    =d_x({_L}\mathcal{D}_{h,1,0,-1}^{\alpha,\lambda_{1},\lambda_{2},\lambda_{3}}+{_R}\mathcal{D}_{h,1,0,-1}^{\alpha,\lambda_{1},\lambda_{2},\lambda_{3}})u_j^{n+1}+f_j^{n+1}+O(\tau+h^2),
\end{equation*}
Dropping the truncation error term, we get the implicit WSGD scheme
\begin{equation}\label{imscheme}
      U_j^{n+1}=U_j^n-\tau\cdot v_x\frac{U_{j+1}^{n+1}-U_{j-1}^{n+1}}{2h}+\frac{d_x\tau}{h^{\alpha}}\sum_{k=0}^{j+1}g_k^{(\alpha)}U_{j-k+1}^{n+1}
        +\frac{d_x\tau}{h^{\alpha}}\sum_{k=0}^{N-j+1}g_k^{(\alpha)}U_{j+k-1}^{n+1}+\tau f_j^{n+1}.
\end{equation}

To improve the accuracy of the numerical scheme in time direction,
using the Crank-Nicolson time discretization, we get the following
Crank-Nicolson-WSGD scheme
\begin{equation}\label{cnscheme}
    \begin{aligned}
      & \frac{U_j^{n+1}-U_{j}^{n}}{\tau}+v_x\frac{U_{j+1}^{n}-U_{j-1}^n}{4h}
      -\left(\frac{d_x}{2h^{\alpha}}\sum_{k=0}^{j+1}g_k^{(\alpha)}U_{j-k+1}^n
        +\frac{d_x}{2h^{\alpha}}\sum_{k=0}^{N-j+1}g_k^{(\alpha)}U_{j+k-1}^n\right) \\
           &=-v_x\frac{U_{j+1}^{n+1}-U_{j-1}^{n+1}}{4h}+\frac{d_x}{2h^{\alpha}}\sum_{k=0}^{j+1}g_k^{(\alpha)}U_{j-k+1}^{n+1}
        +\frac{d_x}{2h^{\alpha}}\sum_{k=0}^{N-j+1}g_k^{(\alpha)}U_{j+k-1}^{n+1}+\frac{\tau}{2} \big(f_{j}^{n+1}+f_{j}^{n}\big).
    \end{aligned}
  \end{equation}

\subsection{Stability analysis}\label{sec:3.1}
To study the stability of the new scheme, we use the von Neumann
linear stability analysis \cite{Chan:84,Strikwerda:04}; and assume
that the solution of the problem (\ref{Problem}) can be zero
extended to the whole real line $\mathbb{R}$. Letting
$\epsilon_{j}^{n}=\tilde{U}_{j}^{n}-U_{j}^{n}$ be the perturbation
error and reasonably replacing the symbols $j+1$ and $N_{x}-j+1$ in
schemes (\ref{imscheme}) and (\ref{cnscheme})  with $\infty$ (since
$\epsilon_{j}^{n}$ should be zero beyond the considered domain),
 we obtain the following error equations
\begin{equation}\label{errimscheme}
      \epsilon_j^{n+1}=\epsilon_j^n-v_x\tau\frac{\epsilon_{j+1}^{n+1}-\epsilon_{j-1}^{n+1}}{2h}+
\frac{d_x\tau}{2h^{\alpha}}\sum_{k=0}^{\infty}g_k^{(\alpha)}\epsilon_{j-k+1}^{n+1}
        +\frac{d_x\tau}{2h^{\alpha}}\sum_{k=0}^{\infty}g_k^{(\alpha)}\epsilon_{j+k-1}^{n+1};
\end{equation}
\begin{equation}\label{errcnscheme}
    \begin{aligned}
      &\epsilon_j^{n+1}=\epsilon_{j}^{n}-v_x\tau\frac{\epsilon_{j+1}^{n}-\epsilon_{j-1}^n}{4h}
      +\left(\frac{d_x\tau}{2h^{\alpha}}\sum_{k=0}^{\infty}g_k^{(\alpha)}\epsilon_{j-k+1}^n
        +\frac{d_x\tau}{2h^{\alpha}}\sum_{k=0}^{\infty}g_k^{(\alpha)}\epsilon_{j+k-1}^n\right) \\
           &~~~~~~~~~ -v_x\tau\frac{\epsilon_{j+1}^{n+1}-\epsilon_{j-1}^{n+1}}{4h}+\frac{d_x\tau}{2h^{\alpha}}
           \sum_{k=0}^{\infty}g_k^{(\alpha)}\epsilon_{j-k+1}^{n+1}
        +\frac{d_x\tau}{2h^{\alpha}}\sum_{k=0}^{\infty}g_k^{(\alpha)}\epsilon_{j+k-1}^{n+1}.
    \end{aligned}
  \end{equation}
\begin{theorem} Let $(\lambda_{1},\lambda_{2},\lambda_{3})$ be chosen in set $\mathcal{S}_{1}$ or $\mathcal{S}_{2}$ or
$\mathcal{S}_{3}$, and assume that the trigonometric polynomial
$Q(\theta,\alpha;\lambda_{1},\lambda_{2},\lambda_{3})\leq0$, for all
$\theta\in[0,\pi]$, $1<\alpha<2$, then the implicit WSGD scheme
(\ref{imscheme}) is unconditionally von Neumann stable, where
$Q(\theta,\alpha;\lambda_{1},\lambda_{2},\lambda_{3})$ is defined by
(\ref{Q}).
\end{theorem}
\begin{proof}
  Let $\epsilon_{j}^{n}=\rho^{n}\mathrm{e}^{ij\theta}$ be the solution of  (\ref{errimscheme}),
  where $i=\sqrt{-1}$, $\rho^n$ is the amplitude at time level $n$ and $\theta(=2\pi h/k)$ is the phase angle with wavelength $k$.
  We need to show that the amplification factor $\rho(\theta)$ satisfies the relation $|\rho(\theta)|\leq1$ for all $\theta$ in $[-\pi,\pi]$. In fact,
  by substituting the expressions of $\epsilon^n_{j}$ and $\epsilon^{n+1}_{j}$ into (\ref{errimscheme}), we get the amplification factor of the implicit WSGD difference scheme
  \begin{equation}
    \rho(\theta)=\frac{1}{1+\frac{\nu}{2}\big(\mathrm{e}^{i\theta}-\mathrm{e}^{-i\theta}\big)-\frac{\lambda}{2}Q(\theta,\alpha;\lambda_{1},\lambda_{2},\lambda_{3})},
  \end{equation}
where $\nu=v_x\tau/h, \lambda=d_x\tau/h^{\alpha}$ and
  \begin{equation}
    Q(\theta,\alpha;\lambda_{1},\lambda_{2},\lambda_{3})=\sum_{k=0}^{\infty}g_k^{(\alpha)}\mathrm{e}^{i(1-k)\theta}
+\sum_{k=0}^{\infty}g_k^{(\alpha)}\mathrm{e}^{-i(1-k)\theta}.
  \end{equation}
In light of the relation (\ref{relation}), we get
  \begin{equation*}
    \begin{aligned}
     Q(\theta,\alpha;\lambda_{1},\lambda_{2},\lambda_{3}) & =\sum_{k=0}^{\infty}g_k^{(\alpha)}\mathrm{e}^{i(1-k)\theta}
              +\sum_{k=0}^{\infty}g_k^{(\alpha)}\mathrm{e}^{-i(1-k)\theta}\\
           & =\lambda_{1}\mathrm{e}^{i\theta}\sum_{k=0}^{\infty}w_k^{(\alpha)}\mathrm{e}^{-ik\theta}
              +\lambda_{2}\sum_{k=0}^{\infty}w_k^{(\alpha)}\mathrm{e}^{-ik\theta}
              +\lambda_{3}\mathrm{e}^{-i\theta}\sum_{k=0}^{\infty}w_k^{(\alpha)}\mathrm{e}^{-ik\theta} \\
           &~~~  +\lambda_{1}\mathrm{e}^{-i\theta}\sum_{k=0}^{\infty}w_k^{(\alpha)}\mathrm{e}^{ik\theta}
              +\lambda_{2}\sum_{k=0}^{\infty}w_k^{(\alpha)}\mathrm{e}^{ik\theta}
              +\lambda_{3}\mathrm{e}^{i\theta}\sum_{k=0}^{\infty}w_k^{(\alpha)}\mathrm{e}^{ik\theta} .
    \end{aligned}
  \end{equation*}
By a straightforward calculation, the trigonometric polynomial $Q(\theta,\alpha;\lambda_{1},\lambda_{2},\lambda_{3})$ gives
   \begin{equation*}
    \begin{aligned}
      Q(\theta,\alpha;\lambda_{1},\lambda_{2},\lambda_{3}) &
      =\lambda_{1}\Big(\mathrm{e}^{i\theta}(1-\mathrm{e}^{-i\theta})^\alpha
              +\mathrm{e}^{-i\theta}(1-\mathrm{e}^{i\theta})^\alpha\Big)
              +\lambda_{2}\Big((1-\mathrm{e}^{-i\theta})^\alpha+(1-\mathrm{e}^{i\theta})^\alpha\Big)\\
            &~~~  +\lambda_{3}\Big(\mathrm{e}^{-i\theta}(1-\mathrm{e}^{-i\theta})^\alpha
              +\mathrm{e}^{i\theta}(1-\mathrm{e}^{i\theta})^\alpha\Big).
    \end{aligned}
  \end{equation*}
In view of $Q(\theta,\alpha;\lambda_{1},\lambda_{2},\lambda_{3})$ is a real-valued
 and even function for given $(\lambda_{1},\lambda_{2},\lambda_{3})$, we just
consider its principal value on $[0,\pi]$. Furthermore, using the relation
\begin{equation}\label{trirelation}
\mathrm{e}^{i\phi}-\mathrm{e}^{i\varphi}=2i\sin\big(\frac{\phi-\varphi}{2}\big)\mathrm{e}^{i\frac{\phi+\varphi}{2}},
\end{equation}
after a straightforward calculation yields
\begin{equation}\label{Q}
Q(\theta,\alpha;\lambda_{1},\lambda_{2},\lambda_{3})=\big(2\sin(\frac{\theta}{2})\big)^\alpha\ \big(\lambda_{1}\cos\big(\frac{\alpha}{2}(\theta-\pi)-\theta\big)
         +\lambda_{2}\cos\big(\frac{\alpha}{2}(\theta-\pi)\big)
         +\lambda_{3}\cos\big(\frac{\alpha}{2}(\theta-\pi)+\theta\big)\big).
\end{equation}
Then the corresponding amplification factor is obtained as
 \begin{equation*}
    \rho(\theta)=\frac{1}{1+i\nu\sin(\theta)-\frac{\lambda}{2} Q(\theta,\alpha;\lambda_{1},\lambda_{2},\lambda_{3})}.
  \end{equation*}
Due to the functions $Q(\theta,\alpha;\lambda_{1},\lambda_{2},\lambda_{3})$ is non-positive, we get
 $$|\rho(\theta)|=\frac{1}{\sqrt{\Big(1-\frac{\lambda}{2} Q(\theta,\alpha;\lambda_{1},\lambda_{2},\lambda_{3})\Big)^2
 +\Big(\nu\sin(\theta)\Big)^2}}\leq1,$$
 which means that
the implicit WSGD difference scheme is unconditionally stable.
\end{proof}

\begin{theorem}
Let $(\lambda_{1},\lambda_{2},\lambda_{3})$ be chosen in set
$\mathcal{S}_{1}$ or $\mathcal{S}_{2}$ or $\mathcal{S}_{3}$, and
assume that the trigonometric polynomial
$Q(\theta,\alpha;\lambda_{1},\lambda_{2},\lambda_{3})\leq0$
 for all $\theta \in[0,\pi],1<\alpha<2$,
then the Crank-Nicolson-WSGD  scheme (\ref{cnscheme}) is unconditionally von Neumann stable,
 where $Q(\theta,\alpha;\lambda_{1},\lambda_{2},\lambda_{3})$ is defined by (\ref{Q}).
\end{theorem}
\begin{proof} Inserting $\epsilon_{j}^{n}=\rho^{n}\mathrm{e}^{ij\theta}\,\,(i=\sqrt{-1})$ into (\ref{errcnscheme}), we get the amplification factor of Crank-Nicolson-WSGD difference scheme
  \begin{equation}
    \rho(\theta)=\frac{1-\frac{\nu}{2}\big(\mathrm{e}^{i\theta}-\mathrm{e}^{-i\theta}\big)+\frac{\lambda}{2}\big(\sum_{k=0}^{\infty}g_k^{(\alpha)}\mathrm{e}^{i(1-k)\theta}
+\sum_{k=0}^{\infty}g_k^{(\alpha)}\mathrm{e}^{-i(1-k)\theta}\big)}{1+\frac{\nu}{2}\big(\mathrm{e}^{i\theta}-\mathrm{e}^{-i\theta}\big)-\frac{\lambda}{2}\big(\sum_{k=0}^{\infty}g_k^{(\alpha)}\mathrm{e}^{i(1-k)\theta}
+\sum_{k=0}^{\infty}g_k^{(\alpha)}\mathrm{e}^{-i(1-k)\theta}\big)},
  \end{equation}
with $\nu=v_x\tau/h, \lambda=d_x\tau/h^{\alpha}$.

By direct calculation, we have
\begin{equation}\label{C}
 \rho(\theta)=
\frac{1-i\nu\sin(\theta)+\frac{\lambda}{2}
Q(\theta,\alpha;\lambda_{1},\lambda_{2},\lambda_{3})}
{1+i\nu\sin(\theta)-\frac{\lambda}{2}
Q(\theta,\alpha;\lambda_{1},\lambda_{2},\lambda_{3})},
\end{equation}
where $Q(\theta,\alpha;\lambda_{1},\lambda_{2},\lambda_{3})$ is
given by (\ref{Q}). Furthermore, using the inequality
$$\frac{(1-a)^2+b^2}{(1+a)^2+b^2}\leq1-\frac{2a}{1+a^2+b^2}\leq1,~a\geq0,~b\in \mathbb{R},$$
and noting that the function
$Q(\theta,\alpha;\lambda_{1},\lambda_{2},\lambda_{3})$ is
non-positive, we have
 $$|\rho(\theta)|^2=\frac{\Big(1+\frac{\lambda }{2}Q(\theta,\alpha;\lambda_{1},\lambda_{2},\lambda_{3})\Big)^2
 +\Big(\nu\sin(\theta)\Big)^2}{\Big(1-\frac{\lambda}{2} Q(\theta,\alpha;\lambda_{1},\lambda_{2},\lambda_{3})\Big)^2
 +\Big(\nu\sin(\theta)\Big)^2}\leq1.$$
\end{proof}
\begin{remark}\label{rea}
There do exist real parameters
$(\lambda_{1},\lambda_{2},\lambda_{3})$ in $\mathcal{S}_{1}$ or
$\mathcal{S}_{2}$ or $\mathcal{S}_{3}$ such that
$Q(\theta,\alpha;\lambda_{1},\lambda_{2},\lambda_{3})\leq0$ holds
for all $\theta \in[0,\pi],1<\alpha<2$. For example, for
$(\lambda_{1},\lambda_{2},\lambda_{3})=(\frac{\alpha}{2},\frac{2-\alpha}{2},0)$,
we have
\begin{equation}\label{Q1}
Q(\theta,\alpha)=\big(2\sin(\frac{\theta}{2})\big)^\alpha\ \big(\frac{\alpha}{2}\cos\big(\frac{\alpha}{2}(\theta-\pi)-\theta\big)
         +\frac{2-\alpha}{2}\cos\big(\frac{\alpha}{2}(\theta-\pi)\big)\big).
         \end{equation}
Here we denote
$Q(\theta,\alpha;\lambda_{1},\lambda_{2},\lambda_{3})$ by
$Q(\theta,\alpha)$ for the simpleness.
 It is easy to check that $Q(\theta,\alpha)$ decreases with respect to $\alpha$, then $Q(\theta,\alpha)\leq Q(\theta,1)=0$.
 And for $(\lambda_{1},\lambda_{2},\lambda_{3})=(\frac{2+\alpha}{4},0,\frac{2-\alpha}{4})$, we have
  \begin{equation}\label{Q2}
  Q(\theta,\alpha)=\bigg(2\sin(\frac{\theta}{2})\bigg)^\alpha\ \big(\frac{\alpha}{2}\sin\big(\frac{\alpha}{2}(\theta-\pi)\big)\sin(\theta)
         +\cos\big(\frac{\alpha}{2}(\theta-\pi)\big)\cos(\theta)\big);
 \end{equation}
it can also be checked that $Q(\theta,\alpha)$ decreases with
respect to $\alpha$, then $Q(\theta,\alpha)\leq
Q(\theta,1)=-2\sin^4(\frac{x}{2})\leq0$. For more general cases, see
the following numerical experiment section.
\end{remark}
\begin{remark}\label{rem:2}
 For $\alpha=2$, the WSGD operator (\ref{eq2.7ww})
  is the centered difference approximation of second order derivative
  when $(\lambda_{1},\lambda_{2},\lambda_{3})$ equals to $(\frac{\alpha}{2},\frac{2-\alpha}{2},0)$
   or $(\frac{2+\alpha}{4},0,\frac{2-\alpha}{4})$. The implicit WSGD scheme (\ref{imscheme})
    reduces to the classical central difference scheme  \cite{Strikwerda:04}.
\end{remark}
It is easy to check that the implicit WSGD scheme (\ref{imscheme})
and the Crank-Nicolson-WSGD scheme (\ref{cnscheme})
 are both consistent with accuracy $O(\tau+h^2)$ and $O(\tau^2+h^2)$, respectively. Using the Lax-Richtmyer equivalence theorem
 \cite{Strikwerda:04},
we get that the difference schemes (\ref{imscheme}) and
(\ref{cnscheme}) are both convergent.
\section{ Extension to two dimensional case}\label{two-dimensions}
We next consider the following two dimensional space fractional
advection diffusion equation
\begin{equation}\label{eq:4.1}
  \begin{cases}
    u_{t}(x,y,t)+v_{x} \partial_{x}u(x,y,t)+v_{y} \partial_{y}u(x,y,t)=d_x\Big({_a}D_x^{\alpha}u(x, y, t)+{_x}D_b^{\alpha}u(x, y, t)\Big)\\
    ~~~~~~~~~~~~~~~~~~  +d_y\Big({_c}D_y^{\beta}u(x, y, t)+{_y}D_d^{\beta}u(x, y, t)\Big)+f(x, y, t),  &   \text{$(x, y, t) \in \Omega\times [0, T]$,} \\
    u(x, y, 0)=u_0(x, y),    & \text{$(x, y)\in \Omega$},\\
    u(x, y, t)=\varphi(x, y, t), &\text{$(x, y, t)\in\partial\Omega\times[0, T]$},
  \end{cases}
\end{equation}
where $\Omega=(a,b)\times(c,d)$, ${_a}D_x^{\alpha}$,
 ${_x}D_b^{\alpha}$,  ${_c}D_y^{\beta}$, and ${_y}D_d^{\beta}$ are
Riemann-Liouville fractional operators with $1<\alpha, \beta \leq
2$. The diffusion coefficients satisfy $d_x, ~d_y\geq 0$ and
$v_x,v_y>0$. We assume that \eqref{eq:4.1} has an unique and
sufficiently smooth solution.
\subsection{Difference schemes}
Now we establish the Crank-Nicolson difference scheme by using WSGD
discretizations \eqref{eq:2.7} for problem \eqref{eq:4.1}. We
partition the domain $\Omega$ into a uniform mesh with the space
stepsizes $h_x=(b-a)/N_x, h_y=(d-c)/N_y$ and the time stepsize
$\tau=\frac{T}{N_{t}}$, where $N_x, N_y, N_t$ are positive integers.
And the set of grid points are denoted by $x_j=jh_x, y_m=mh_y$ and
$t_n=n\tau$ for $1\leq j\leq N_x, 1\leq m\leq N_y$ and $0\leq n\leq
N_{t}$. Let $t_{n+1/2}=(t_n+t_{n+1})/2$ for $0\leq n\leq N_{t}-1$,
and we use the following notations
\begin{equation*}
  u_{j,m}^n=u(x_j, y_m, t_n), \ \ f_{i, j}^{n+1/2}=f(x_i, y_j, t_{n+1/2}), \ \ \delta_x u_{j,m}^n=\frac{u_{j+1, m}^{n}-u_{j-1, m}^n}{2h_{x}}, \ \ \delta_y u_{j,m}^n=\frac{u_{j, m+1}^{n}-u_{j, m-1}^n}{2h_{y}}.
\end{equation*}

Discretizing \eqref{eq:4.1} in time direction by the Crank-Nicolson
scheme,  we get
\begin{equation}\label{eq:4.2}
  \begin{split}
    \frac{u_{j,m}^{n+1}-u_{j,m}^n}{\tau}&=-\frac{1}{2}\Big(v_x(\partial_{x}u)_{j,m}^{n+1}+v_y(\partial_yu)_{j,m}^{n+1}
    +v_x(\partial_{x}u)_{j,m}^{n}+v_y(\partial_{y}u)_{j,m}^{n}\Big)\\
    &+\frac{1}{2}\Big(d_x(_aD_x^{\alpha}u)_{j,m}^{n+1}+d_y(_xD_b^{\alpha}u)_{j,m}^{n+1}
    +d_x(_cD_y^{\beta}u)_{j,m}^{n+1}+d_y(_yD_d^{\beta}u)_{j,m}^{n+1}                           \\
    &+d_x(_aD_x^{\alpha}u)_{j,m}^n+d_y(_xD_b^{\alpha}u)_{j,m}^n+d_x(_cD_y^{\beta}u)_{j,m}^n+d_y(_yD_d^{\beta}u)_{j,m}^n\Big)+f_{j,m}^{n+1/2}+O(\tau^2).
  \end{split}
\end{equation}
In space discretizations, we choose the WSGD operators
${_L}\mathcal{D}_{h_{x},1,0,-1}^{\alpha,\lambda_{1},\lambda_{2},\lambda_{3}}u$,
${_R}\mathcal{D}_{h_{x},1,0,-1}^{\alpha,\lambda_{1},\lambda_{2},\lambda_{3}}u$,
${_L}\mathcal{D}_{h_{y},1,0,-1}^{\beta,\lambda_{1},\lambda_{2},\lambda_{3}}u$,
and
${_R}\mathcal{D}_{h_{y},1,0,-1}^{\beta,\lambda_{1},\lambda_{2},\lambda_{3}}u$
to respectively approximate the fractional diffusion terms
$_aD_x^{\alpha}u, ~_xD_b^{\alpha}u$, $_cD_y^{\beta}u$, and
$_yD_d^{\beta}u$. And multiplying \eqref{eq:4.2} by $\tau$ and
separating the time layers, we have that
\begin{equation}\label{eq:4.3}
  \begin{split}
    &\Big(1-\frac{d_x\tau}{2}{_L}\mathcal{D}_{h_{x},1,0,-1}^{\alpha,\lambda_{1},\lambda_{2},\lambda_{3}}
        -\frac{d_y\tau}{2}{_R}\mathcal{D}_{h_{x},1,0,-1}^{\alpha,\lambda_{1},\lambda_{2},\lambda_{3}}
        -\frac{d_x\tau}{2}{_L}\mathcal{D}_{h_{y},1,0,-1}^{\beta,\lambda_{1},\lambda_{2},\lambda_{3}}
        -\frac{d_y\tau}{2}{_R}\mathcal{D}_{h_{y},1,0,-1}^{\alpha,\lambda_{1},\lambda_{2},\lambda_{3}}
        +\frac{v_x\tau}{2}\delta_{x}+\frac{v_y\tau}{2}\delta_{y}\Big)u_{j,m}^{n+1} \\
    &=\Big(1+\frac{d_x\tau}{2}{_L}\mathcal{D}_{h_{x},1,0,-1}^{\alpha,\lambda_{1},\lambda_{2},\lambda_{3}}
        +\frac{d_y\tau}{2}{_R}\mathcal{D}_{h_{x},1,0,-1}^{\alpha,\lambda_{1},\lambda_{2},\lambda_{3}}
        +\frac{d_x\tau}{2}{_L}\mathcal{D}_{h_{y},1,0,-1}^{\beta,\lambda_{1},\lambda_{2},\lambda_{3}}
        +\frac{d_y\tau}{2}{_R}\mathcal{D}_{h_{y},1,0,-1}^{\alpha,\lambda_{1},\lambda_{2},\lambda_{3}}
        -\frac{v_x\tau}{2}\delta_{x}-\frac{v_y\tau}{2}\delta_{y}\Big)u_{j,m}^{n}
        \\
    &+\tau f_{j,m}^{n+1/2}+\tau T_{j,m}^{n},
  \end{split}
\end{equation}
where $|T_{j,m}^{n}|\leq \tilde{c}(\tau^2+h^2)$ denotes the
truncation error. And we denote
\begin{equation*}
  \delta_x^\alpha=d_x({_L}\mathcal{D}_{h_{x},1,0,-1}^{\alpha,\lambda_{1},\lambda_{2},\lambda_{3}}
  +{_R}\mathcal{D}_{h_{x},1,0,-1}^{\alpha,\lambda_{1},\lambda_{2},\lambda_{3}})+v_{x}\delta_{x}, \qquad
  \delta_y^\beta=d_y({_L}\mathcal{D}_{h_{y},1,0,-1}^{\beta,\lambda_{1},\lambda_{2},\lambda_{3}}
  +{_R}\mathcal{D}_{h_{y},1,0,-1}^{\beta,\lambda_{1},\lambda_{2},\lambda_{3}})+v_{y}\delta_{y}.
\end{equation*}
For the simplicity, the stepsizes are chosen as the same, i.e.,
$h=h_x=h_y$. Using the Taylor-series expansions, with the similar
argument presented in \cite{Meerschaert:06b}, we can check that
\begin{equation}\label{eq:4.4}
  \begin{split}
    \frac{\tau^2}{4}\delta_x^{\alpha}\delta_x^{\beta}(u_{j,m}^{n+1}-u_{j,m}^{n})
    &=\frac{\tau^2}{4}\delta_x^{\alpha}\delta_x^{\beta}(u^{n+1/2}_t)+O(\tau^5)
        \\
    &=\frac{\tau^3}{4}\Big(d_x({_a}D_x^{\alpha}+{_x}D_{b}^{\alpha}-\partial_{x})d_y({_c}D_{y}^{\beta}
 +{_y}D_{d}^{\beta}-\partial_{y})u_t\Big)_{j,m}^{n+1/2}+O(\tau^5+\tau^3h^2).
  \end{split}
\end{equation}
Adding \eqref{eq:4.4} to the right-hand side of \eqref{eq:4.3} and
making the factorization leads to
\begin{equation}
  \begin{split}\label{eq:4.6}
    \Big(1-\frac{\tau}{2}\delta_x^\alpha\Big)\Big(1-\frac{\tau}{2}\delta_y^\beta\Big)u_{j,m}^{n+1}
    = \Big(1+\frac{\tau}{2}\delta_x^\alpha\Big)\Big(1+\frac{\tau}{2}\delta_y^\beta\Big)u_{j,m}^n
    +\tau f_{i,j}^{n+1/2}+\tau T_{j,m}^{n}+O(\tau^3+\tau^3h^2).
  \end{split}
\end{equation}
Dropping the truncation error terms, we obtain the finite difference
approximation
\begin{equation}
  \begin{split}\label{eq:4.7}
    & \Big(1-\frac{\tau}{2}\delta_x^\alpha\Big)\Big(1-\frac{\tau}{2}\delta_y^\beta\Big)U_{j,m}^{n+1}
    = \Big(1+\frac{\tau}{2}\delta_x^\alpha\Big)\Big(1+\frac{\tau}{2}\delta_y^\beta\Big)U_{j,m}^n+\tau f_{j,m}^{n+1/2}.
  \end{split}
\end{equation}
For efficiently solving (\ref{eq:4.7}), we use the following three
splitting schemes:

 Peaceman-Rachford ADI scheme \cite{Peaceman:59}:
\begin{subequations}\label{eq:4.16}
\begin{align}
    &\Big(1-\frac{\tau}{2}\delta_x^\alpha\Big)V_{j,m}^{n}~~
     =\Big(1+\frac{\tau}{2}\delta_y^\beta\Big)U_{j,m}^{n}+\frac{\tau}{2}f_{j,m}^{n+1/2}, \\
    &\Big(1-\frac{\tau}{2}\delta_y^\beta\Big)U_{j,m}^{n+1}
     =\Big(1+\frac{\tau}{2}\delta_x^\alpha\Big)V_{j,m}^{n}+\frac{\tau}{2}f_{j,m}^{n+1/2}.
\end{align}
\end{subequations}
Douglas ADI scheme \cite{Sun:05}:
\begin{subequations}\label{eq:4.17}
\begin{align}
    &\Big(1-\frac{\tau}{2}\delta_x^\alpha\Big)V_{j,m}^{n}~~=
     \Big(1+\frac{\tau}{2}\delta_x^\alpha+\tau\delta_y^\beta\Big)U_{j,m}^{n}+\tau f_{j,m}^{n+1/2}, \\
    &\Big(1-\frac{\tau}{2}\delta_y^\beta\Big)U_{j,m}^{n+1}
     =V_{j,m}^n-\frac{\tau}{2}\delta_y^\beta U_{j,m}^n.
\end{align}
\end{subequations}
D'Yakonov ADI scheme \cite{Yakonov:64}:
\begin{subequations}\label{eq:4.18}
\begin{align}
    &\Big(1-\frac{\tau}{2}\delta_x^\alpha\Big)V_{j,m}^{n}~~=\Big(1+\frac{\tau}{2}\delta_x^\alpha\Big)
     \Big(1+\frac{\tau}{2}\delta_y^\beta\Big)U_{j,m}^{n}+\tau f_{j,m}^{n+1/2}, \\
    &\Big(1-\frac{\tau}{2}\delta_y^\beta\Big)U_{j,m}^{n+1}=V_{j,m}^n.
\end{align}
\end{subequations}
\subsection{Stability}
Now we consider the stability and convergence analysis for the
CN-WSGD scheme \eqref{eq:4.7}. To study the stability of the new scheme, we also use the von Neumann linear stability analysis.
\begin{theorem}\label{thm:6}
Let $(\lambda_{1},\lambda_{2},\lambda_{3})$ and $(\gamma_{1},\gamma_{2},\gamma_{3})$ be chosen in set $\mathcal{S}_{1}$ or $\mathcal{S}_{2}$ or
$\mathcal{S}_{3}$,
and assume that the trigonometric polynomials $Q(\theta_{x},\alpha;\lambda_{1},\lambda_{2},\lambda_{3})\leq0,
~Q(\theta_{x},\beta;\gamma_{1},\gamma_{2},\gamma_{3})\leq0 $,
 for all $\theta_x, \theta_y\in[0,\pi],1<\alpha,\beta<2$, then the difference scheme \eqref{eq:4.7} is unconditionally stable for $1<\alpha,\beta\le2$,
where $Q(\theta,\alpha;\lambda_{1},\lambda_{2},\lambda_{3})$ and $Q(\theta_{x},\beta;\gamma_{1},\gamma_{2},\gamma_{3})$ are defined by (\ref{Qx}) and (\ref{Qy}), respectively.
\end{theorem}
\begin{proof}
  Let  $\epsilon_{j,m}^{n}=\rho^{n}\mathrm{e}^{i\theta_{x}j+i\theta_{y}m}$ where $i=\sqrt{-1}$, $\rho^{n}$ is the amplitude at time level $n$; $\theta_{x}=2\pi h_{x}/k_{x}$ and $\theta_{y}=2\pi h_{y}/k_{y}$ are the phase angles with wavelengths $k_{x}$ and $k_{y}$, respectively; the amplification
factor $G(\theta_{x},\theta_{y})=\rho^{n+1}/\rho^{n}$; so if $|G(\theta_{x},\theta_{y})|\leq1$ for all $\theta_{x}$ and $\theta_{y}$ in $[-\pi,\pi]$, then the numerical scheme is stable. Substituting the expressions of
 $\epsilon_{j,m}^{n}$ and $\epsilon_{j,m}^{n+1}$ into the following error equation
 \begin{equation}
  \begin{split}\label{eq:4.7w}
    & \Big(1-\frac{\tau}{2}\delta_x^\alpha\Big)\Big(1-\frac{\tau}{2}\delta_y^\beta\Big)\epsilon_{j,m}^{n+1}
    = \Big(1+\frac{\tau}{2}\delta_x^\alpha\Big)\Big(1+\frac{\tau}{2}\delta_y^\beta\Big)\epsilon_{j,m}^n,
  \end{split}
\end{equation}
then we have the amplification factor
\begin{equation}\label{eq:4.13}
    G(\theta_{x},\theta_{y})=
\frac{\Big(1-i\nu_{x}\sin(\theta_{x})+\frac{\lambda_{x}}{2} Q(\theta_{x},\alpha;\lambda_{1},\lambda_{2},\lambda_{3})\Big)
\Big(1-i\nu_{y}\sin(\theta_{y})+\frac{\lambda_{y}}{2} Q(\theta_{y},\beta;\gamma_{1},\gamma_{2},\gamma_{3})\Big)}
{\Big(1+i\nu_{x}\sin(\theta_{x})-\frac{\lambda_{x}}{2} Q(\theta_{x},\alpha;\lambda_{1},\lambda_{2},\lambda_{3})\Big)
\Big(1+i\nu_{y}\sin(\theta_{y})-\frac{\lambda_{y}}{2} Q(\theta_{y},\beta;\gamma_{1},\gamma_{2},\gamma_{3})\Big)},
  \end{equation}
  where
  \begin{equation}\label{Qx}
Q(\theta_{x},\alpha;\lambda_{1},\lambda_{2},\lambda_{3})=\big(2\sin(\frac{\theta_{x}}{2})\big)^\alpha\ \big(\lambda_{1}\cos\big(\frac{\alpha}{2}(\theta_{x}-\pi)-\theta_{x}\big)
         +\lambda_{2}\cos\big(\frac{\alpha}{2}(\theta_{x}-\pi)\big)
         +\lambda_{3}\cos\big(\frac{\alpha}{2}(\theta_{x}-\pi)+\theta_{x}\big)\big),
\end{equation}
and
\begin{equation}\label{Qy}
Q(\theta_{y},\beta;\gamma_{1},\gamma_{2},\gamma_{3})=\big(2\sin(\frac{\theta_{y}}{2})\big)^\beta\ \big(\gamma_{1}\cos\big(\frac{\beta}{2}(\theta_{y}-\pi)-\theta_{y}\big)
         +\gamma_{2}\cos\big(\frac{\beta}{2}(\theta_{y}-\pi)\big)
         +\gamma_{3}\cos\big(\frac{\beta}{2}(\theta_{y}-\pi)+\theta_{y}\big)\big).
\end{equation}
If the parameters $(\gamma_{1},\gamma_{2},\gamma_{3})$ and $(\lambda_{1},\lambda_{2},\lambda_{3})$ are chosen such that $Q(\theta_{x},\alpha;\lambda_{1},\lambda_{2},\lambda_{3})$ and $Q(\theta_{y},\beta;\gamma_{1},\gamma_{2},\gamma_{3})$ are all non-positive, then  $|\rho(\theta_{x})|\leq1$ and $|\rho(\theta_{y})|\leq1$ for all $\theta_{x}\in[0,\pi]$ and $\theta_{y}\in[0,\pi]$, respectively.
So it holds that
\begin{equation}\label{Qxy}
|G(\theta_{x},\theta_{y})|\leq|\rho(\theta_{x})||\rho(\theta_{y})|\leq1,
\end{equation}
 which means that \eqref{eq:4.7} is unconditionally stable.
\end{proof}

\section{Numerical Experiments}\label{Numericalsec}
In this section, we demonstrate the performance of
the proposed new second order WSGD schemes for space fractional advection diffusion equation. The numerical errors are measured by the
pointwise maximum norm,
$$\|e\|_{\infty}=\max_{1\leq n\leq N_{t}-1}\{\max_{1\leq j\leq N_{x}-1}\max_{\,\,1\leq m\leq N_{y}-1}\{|u(x_j,y_m,t_n)-U_{j,m}^{n}|\}\}.$$
To confirm the convergent orders and show the accuracy of the numerical schemes, both one and two dimensional space fractional advection diffusion equations are numerically computed. As we have discussed in the previous sections,
the implicit WSGD scheme or Crank-Nicolson-WSGD scheme is
unconditionally stable if the trigonometric polynomials
$Q(\theta,\alpha)$ (for the simplicity, both $Q(\theta,\alpha;\lambda_{1},\lambda_{2},\lambda_{3})$ and $Q(\theta,\alpha;\gamma_{1},\gamma_{2},\gamma_{3})$ are breviated as $Q(\theta,\alpha)$) are non-positive for any $\theta \in [0,\pi]$.
However, it seems not easy to analytically get the effective regions for the parameters $(\lambda_{1},\lambda_{2},\lambda_{3})$ or $(\gamma_{1},\gamma_{2},\gamma_{3})$ from $Q(\theta,\alpha)$. Fortunately, these can be easily done by numerical test. For any $\alpha \in (1,2)$, by numerical test, we find that $Q(\theta,\alpha)$ is non-positive when $\lambda_3\in [-0.005,2.0]$ in set
$\mathcal{S}_{3}$ or $\lambda_2\in [-4,0.5]$ in set $\mathcal{S}_{2}$ or $\lambda_1\in [0.75,3]$ in set $\mathcal{S}_{1}$.  In the
following numerical computations, we choose the parameters
$(\lambda_{1},\lambda_{2},\lambda_{3})$ or $(\gamma_{1},\gamma_{2},\gamma_{3})$ in the above intervals.  The maximum norm errors and their convergence rates to
Example \ref{exm:1} approximated by the implicit WSGD
 scheme and Crank-Nicolson-WSGD  scheme are presented in Tables \ref{tab:1}-\ref{tab:6}.


\begin{example}\label{exm:1}
   We first consider the following one dimensional problem
  \begin{equation}\label{exm21}
    \begin{split}
      & u_{t}(x,t)+u_{x}(x,t)={_0}D_x^{\alpha}u(x,t)+{_x}D_1^{\alpha}u(x,t)
        +f(x,t), \quad (x,t)\in (0,1)\times(0,1],\\
      & u(0,t)=u(1,t)=0, \quad t\in (0, 1],\\
      & u(x,0)=x^3(2-x)^3, \quad x\in (0, 1),
    \end{split}
  \end{equation}
  with the source term
  \begin{equation*}
    \begin{split}
    f(x,t)=\mathrm{e}^{-t}\Big(-x^3(1-x)^3+\big(3x^2(1-x)^3-3x^3(1-x)^2\big)-\frac{\Gamma(4)}{\Gamma(4-\alpha)}\big(x^{3-\alpha}+(1-x)^{3-\alpha}\big)
                    \\
                   +\frac{3\Gamma(5)}{\Gamma(5-\alpha)}\big(x^{4-\alpha}+(1-x)^{4-\alpha}\big) -\frac{3\Gamma(6)}{\Gamma(6-\alpha)}\big(x^{5-\alpha}+(1-x)^{5-\alpha}\big)
                     +\frac{\Gamma(7)}{\Gamma(7-\alpha)}\big(x^{6-\alpha}+(1-x)^{6-\alpha}\big)\Big).
    \end{split}
  \end{equation*}
  It is easy to check that the exact solution of (\ref{exm21}) is $u(x,t)=\mathrm{e}^{-t}x^3(1-x)^3$.
\end{example}
\begin{figure}[h]
\centering
    \subfigure[$\lambda_{1}=0.75$]{
        \label{fig:mini:subfig:a} 
        \begin{minipage}[b]{0.35\textwidth}
            \centering
            \includegraphics[scale=.4]{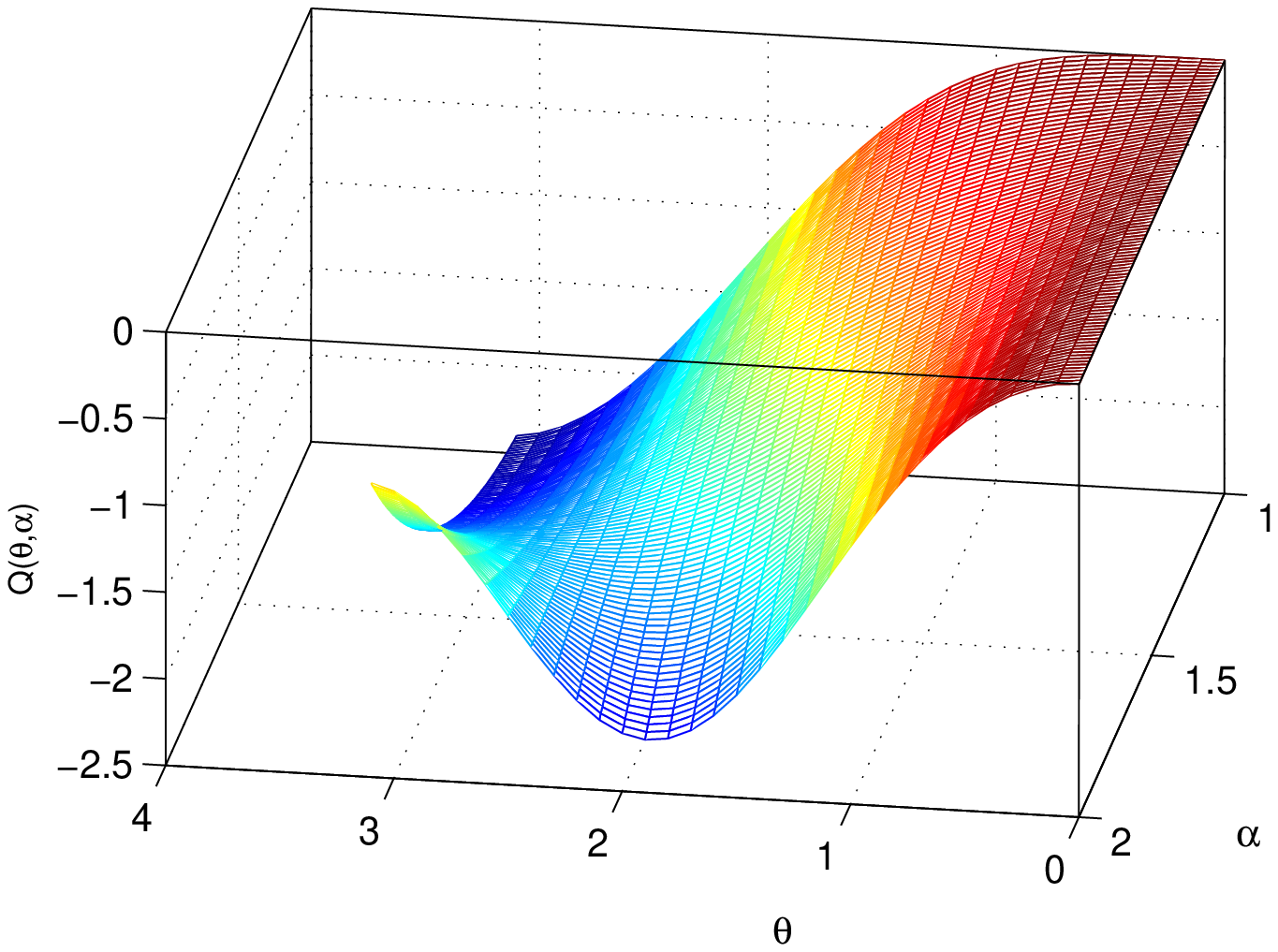}
        \end{minipage}}
\subfigure[$\lambda_{1}=1$]{
        \label{fig:mini:subfig:a} 
        \begin{minipage}[b]{0.35\textwidth}
            \centering
            \includegraphics[scale=.4]{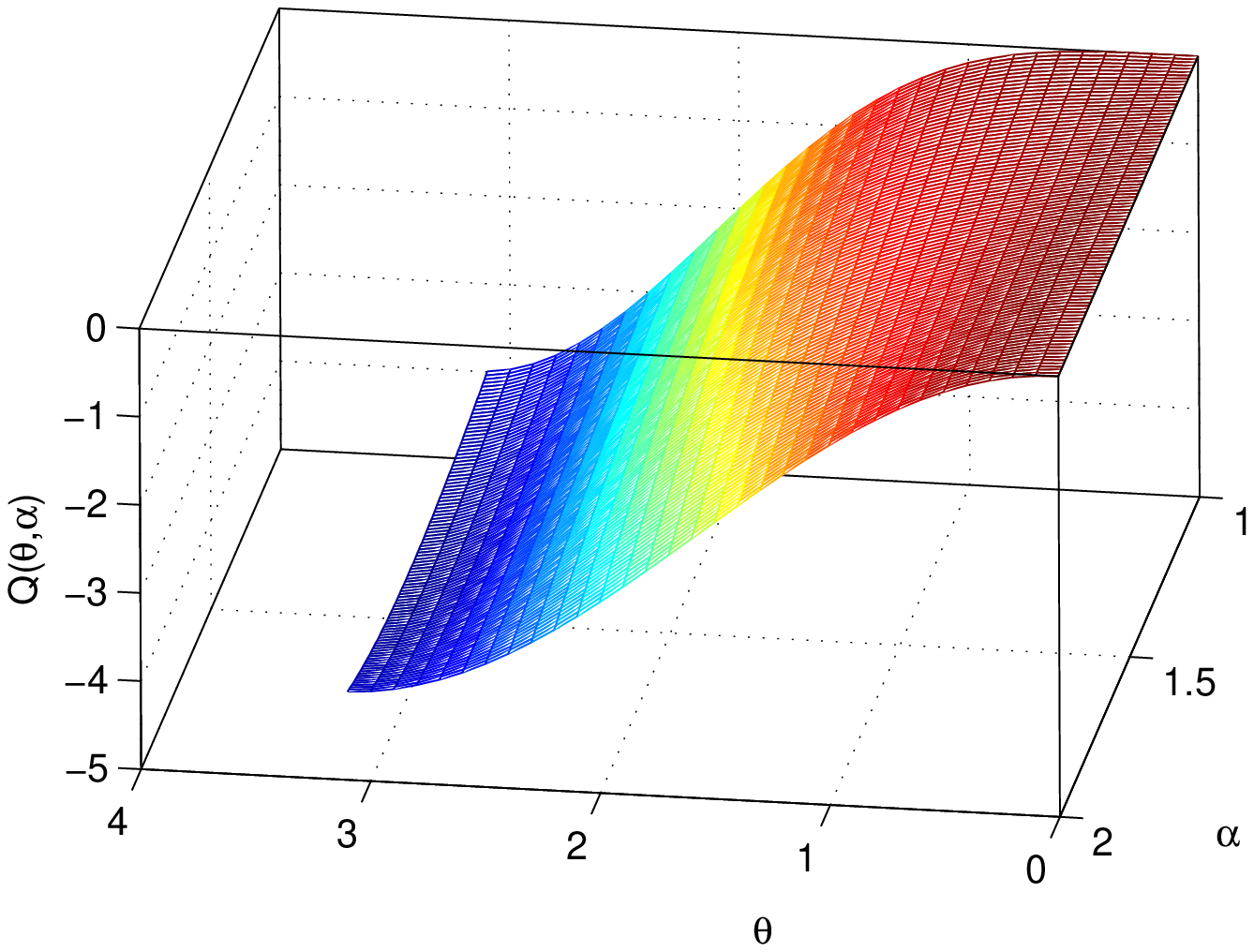}
        \end{minipage}}\\
    \subfigure[$\lambda_{1}=1.2$]{
        \label{fig:mini:subfig:b} 
        \begin{minipage}[b]{0.35\textwidth}
            \centering
            \includegraphics[scale=.4]{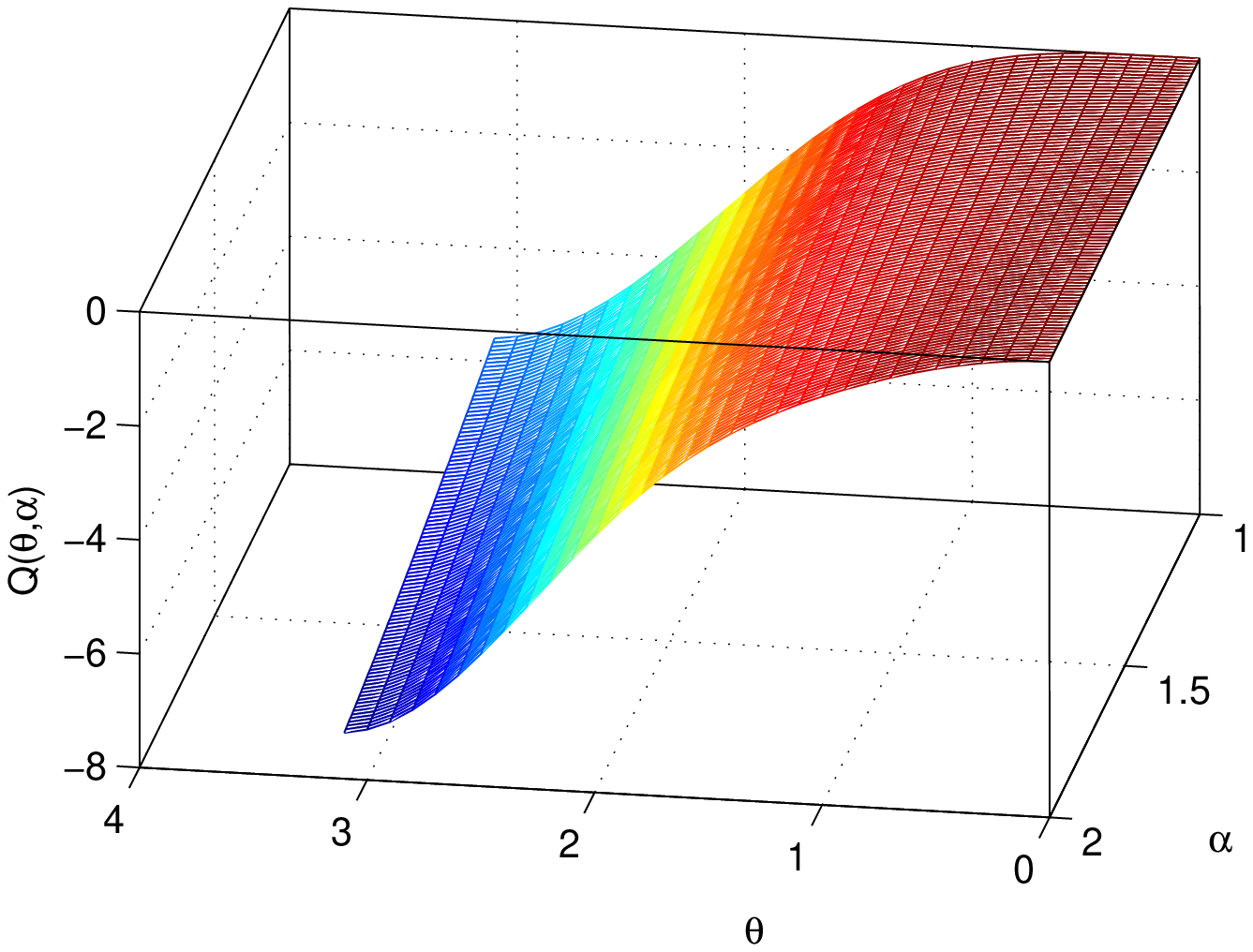}
        \end{minipage}}
    \subfigure[$\lambda_{1}=1.5$]{
        \label{fig:mini:subfig:a} 
        \begin{minipage}[b]{0.35\textwidth}
            \centering
            \includegraphics[scale=.4]{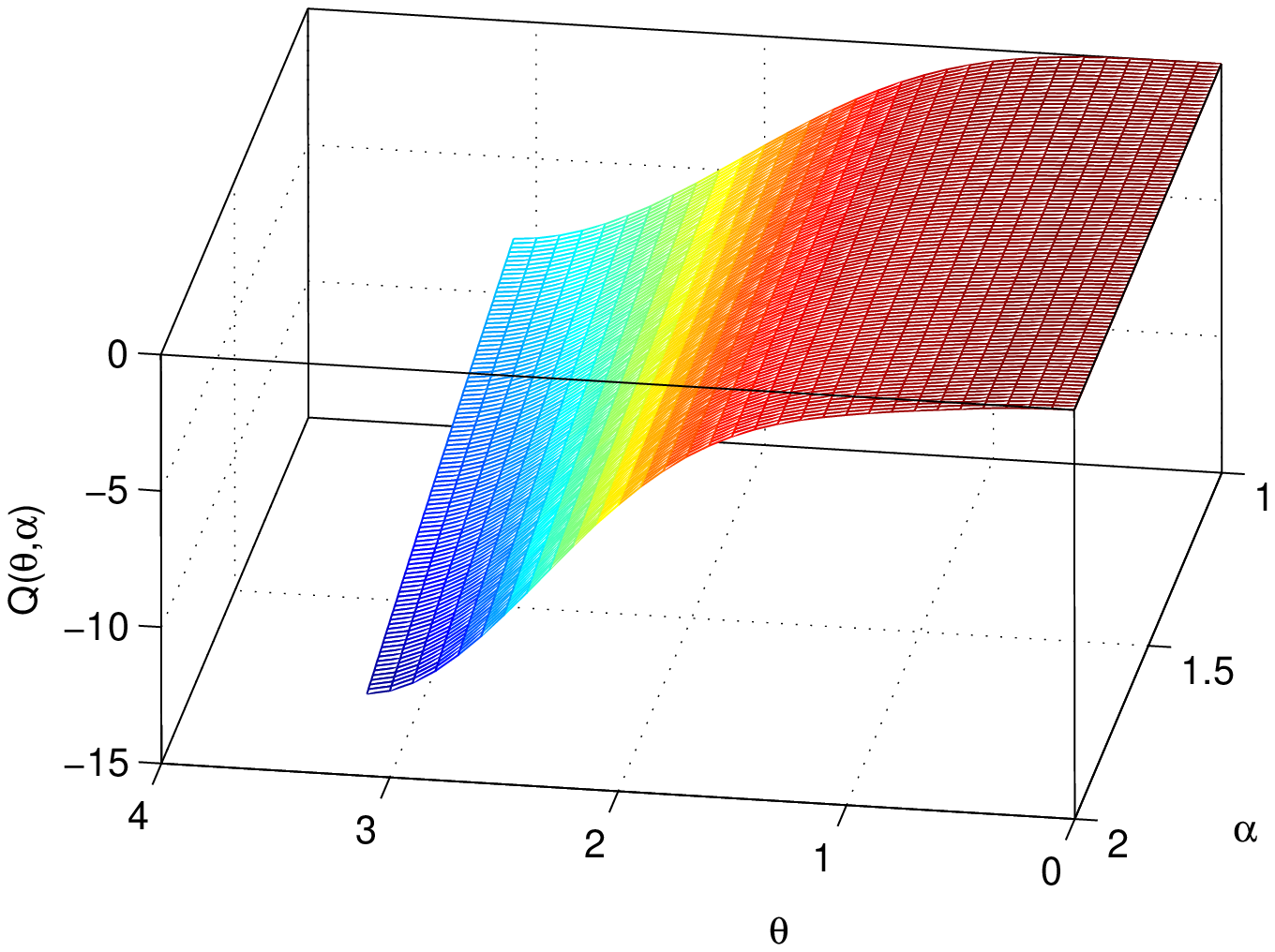}
        \end{minipage}}
    \caption{The trigonometric polynomials $Q(\theta,\alpha)$ in amplification factors $\rho(\theta)$ and the parameters $(\lambda_1,\lambda_2,\lambda_3)$ selected in set
     $\mathcal{S}_{1}=\Big\{~\lambda_1~ \textrm{is~given},\lambda_2=\frac{2+\alpha}{2}-2\lambda_{1},~\lambda_3=\lambda_1-\frac{\alpha}{2}\Big\}$.}
    \label{fig1} 
\end{figure}

\begin{table}[H]\fontsize{10pt}{12pt}\selectfont
  \begin{center}
  \caption{Numerical errors of Example \ref{exm:1} calculated by the implicit WSGD scheme
  at $t=1$ for different $\alpha$ with fixed time stepsize $\tau=h^2$ and the parameters $(\lambda_{1},\lambda_{2},\lambda_{3})$
  selected in set $\mathcal{S}_{1}$.}\vspace{5pt}
  \begin{tabular*}{\linewidth}{@{\extracolsep{\fill}}*{2}{r}*{10}{c}}
    \toprule
     & & \multicolumn{2}{c}{$\lambda_{1}=0.75$} & \multicolumn{2}{c}{$\lambda_{1}=1.0$}& \multicolumn{2}{c}{$\lambda_{1}=1.2$}& \multicolumn{2}{c}{$\lambda_{1}=1.5$} \\
    \cline{3-4} \cline{5-6} \cline{7-8}\cline{9-10}\\[-8pt]
    $\alpha$ & $h$ & $\|e\|_{\infty}$ & Rate
    & $\|e\|_{\infty}$ & Rate & $\|e\|_{\infty}$ & Rate& $\|e\|_{\infty}$ & Rate\\
    \toprule
     & 1/10 & 5.63E-04 & -    & 6.96E-04 & -    & 8.42E-04 & -   &1.02E-03&-\\
 1.2 & 1/20 & 1.61E-04 & 1.80 & 2.10E-04 & 1.73 & 2.54E-04 & 1.73&3.17E-04&1.69\\
     & 1/40 & 4.41E-05 & 1.87 & 6.52E-05 & 1.69 & 8.31E-05 & 1.61&1.11E-04&1.52\\
    \toprule
     & 1/10 & 4.67E-04 & -    & 8.43E-04 & -   & 1.38E-03 & -   &2.25E-03&-\\
1.7  & 1/20 & 6.60E-05 & 2.82 & 2.26E-04 & 1.90& 3.85E-04 & 1.84&6.32E-04&1.83\\
     & 1/40 & 1.27E-05 & 2.37 & 5.65E-05 & 2.00& 9.60E-05 & 2.00&1.56E-04&2.02 \\
    \toprule
     & 1/10 & 4.69E-004 & -    & 5.35E-04 & -    & 1.23E-03 & -   &2.43E-03&-\\
1.9  & 1/20 & 9.92E-005 & 2.15 & 1.33E-04 & 2.01 & 3.18E-04 & 1.96&6.05E-04&2.00  \\
     & 1/40 & 2.35E-005 & 2.13 & 3.25E-05 & 2.03 & 7.66E-05 & 2.06&1.42E-04&2.09  \\
    \toprule
  \end{tabular*}\label{tab:1}
  \end{center}
\end{table}
\begin{table}[H]\fontsize{10pt}{12pt}\selectfont
  \begin{center}
  \caption{Numerical errors of Example \ref{exm:1} calculated by the Crank-Nicolson-WSGD
   scheme at $t=1$ for different $\alpha$ with fixed time stepsize $\tau=h$ and the parameters
   $(\lambda_{1},\lambda_{2},\lambda_{3})$ selected in set $\mathcal{S}_{1}$.}\vspace{5pt}
  \begin{tabular*}{\linewidth}{@{\extracolsep{\fill}}*{2}{r}*{10}{c}}
    \toprule
     & & \multicolumn{2}{c}{$\lambda_{1}=0.75$}& \multicolumn{2}{c}{$\lambda_{1}=1.0$} & \multicolumn{2}{c}{$\lambda_{1}=1.2$}& \multicolumn{2}{c}{$\lambda_{1}=1.5$} \\
    \cline{3-4} \cline{5-6}\cline{7-8}\cline{9-10} \\[-8pt]
    $\alpha$ & $h$ & $\|e\|_{\infty}$ & Rate
    & $\|e\|_{\infty}$ & Rate & $\|e\|_{\infty}$ & Rate& $\|e\|_{\infty}$ & Rate\\
    \toprule
     & 1/10 & 5.63E-04 & -    & 7.04E-04 & -    &8.56E-04&-   & 1.04E-03 & -\\
 1.2 & 1/20 & 1.60E-04 & 1.82 & 2.08E-04 & 1.76 &2.51E-04&1.77& 3.15E-04 & 1.72\\
     & 1/40 & 4.34E-05 & 1.88 & 6.46E-05 & 1.68 &8.25E-05&1.61& 1.10E-04 & 1.52\\
    \toprule
     & 1/10 & 7.54E-04 & -     & 1.09E-03 & -   &1.71E-03&-   & 2.70E-03 & - \\
1.7  & 1/20 & 9.81E-05 & 2.94  & 2.42E-04 & 2.17&4.02E-04&2.09& 6.48E-04 & 2.06\\
     & 1/40 & 1.38E-05 & 2.82  & 5.62E-05 & 2.10&9.58E-05&2.07& 1.56E-04 & 2.06 \\
    \toprule
     & 1/10 & 7.46E-004 & -    & 7.89E-04 & -   &1.72E-03&-   & 3.29E-03 & -   \\
1.9  & 1/20 & 1.07E-004 & 2.80 & 1.69E-04 & 2.22&3.85E-04&2.16& 7.20E-04 & 2.19  \\
     & 1/40 & 2.39E-005 & 2.16 & 3.95E-05 & 2.10&8.47E-05&2.19& 1.57E-04 & 2.20  \\
    \toprule
  \end{tabular*}\label{tab:2}
  \end{center}
\end{table}

\begin{figure}[h]
\centering
    \subfigure[$\lambda_{2}=-2$]{
        \label{fig:mini:subfig:a} 
        \begin{minipage}[b]{0.35\textwidth}
            \centering
            \includegraphics[scale=.4]{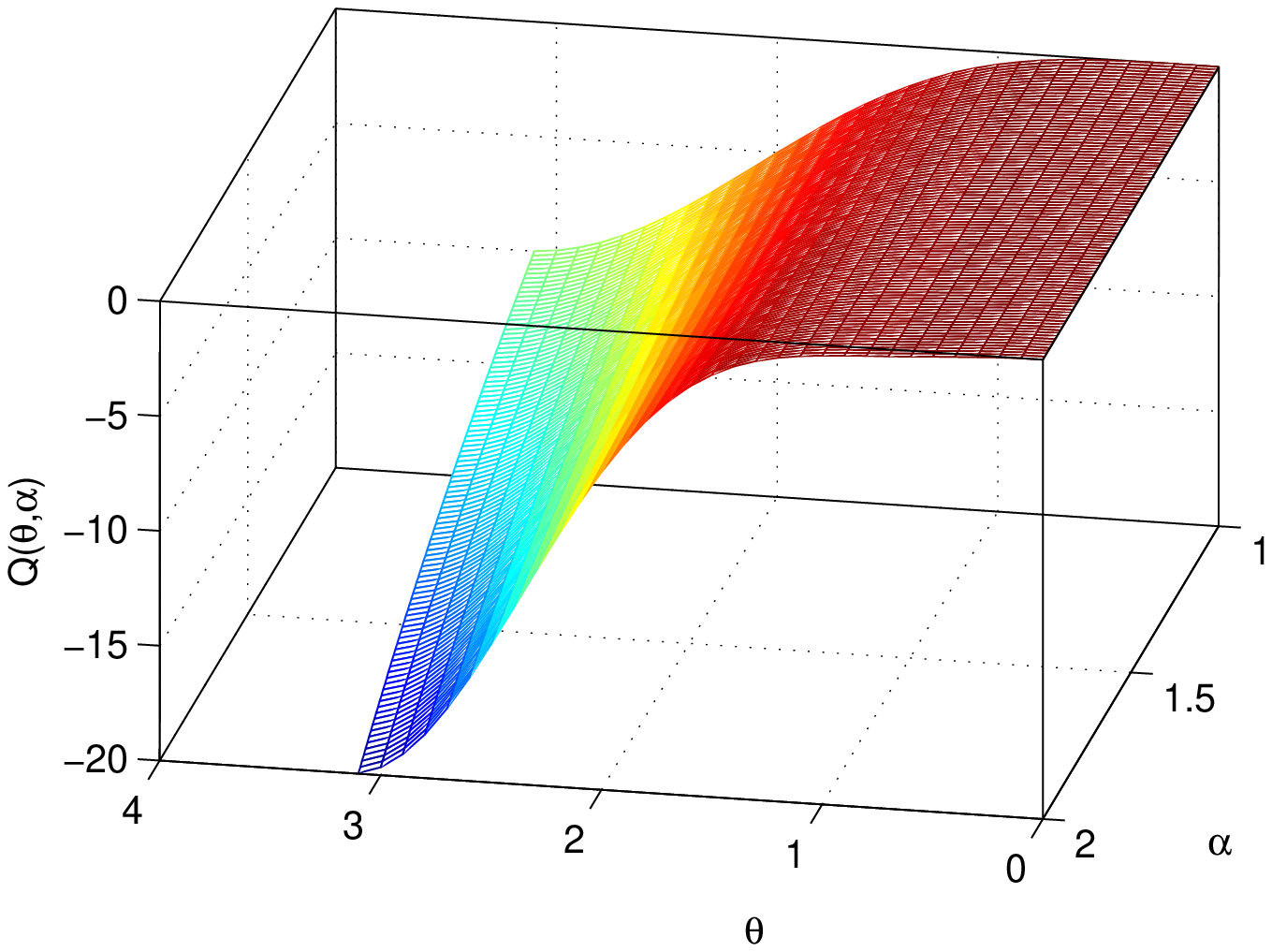}
        \end{minipage}}
\subfigure[$\lambda_{2}=-0.1$]{
        \label{fig:mini:subfig:a} 
        \begin{minipage}[b]{0.35\textwidth}
            \centering
            \includegraphics[scale=.4]{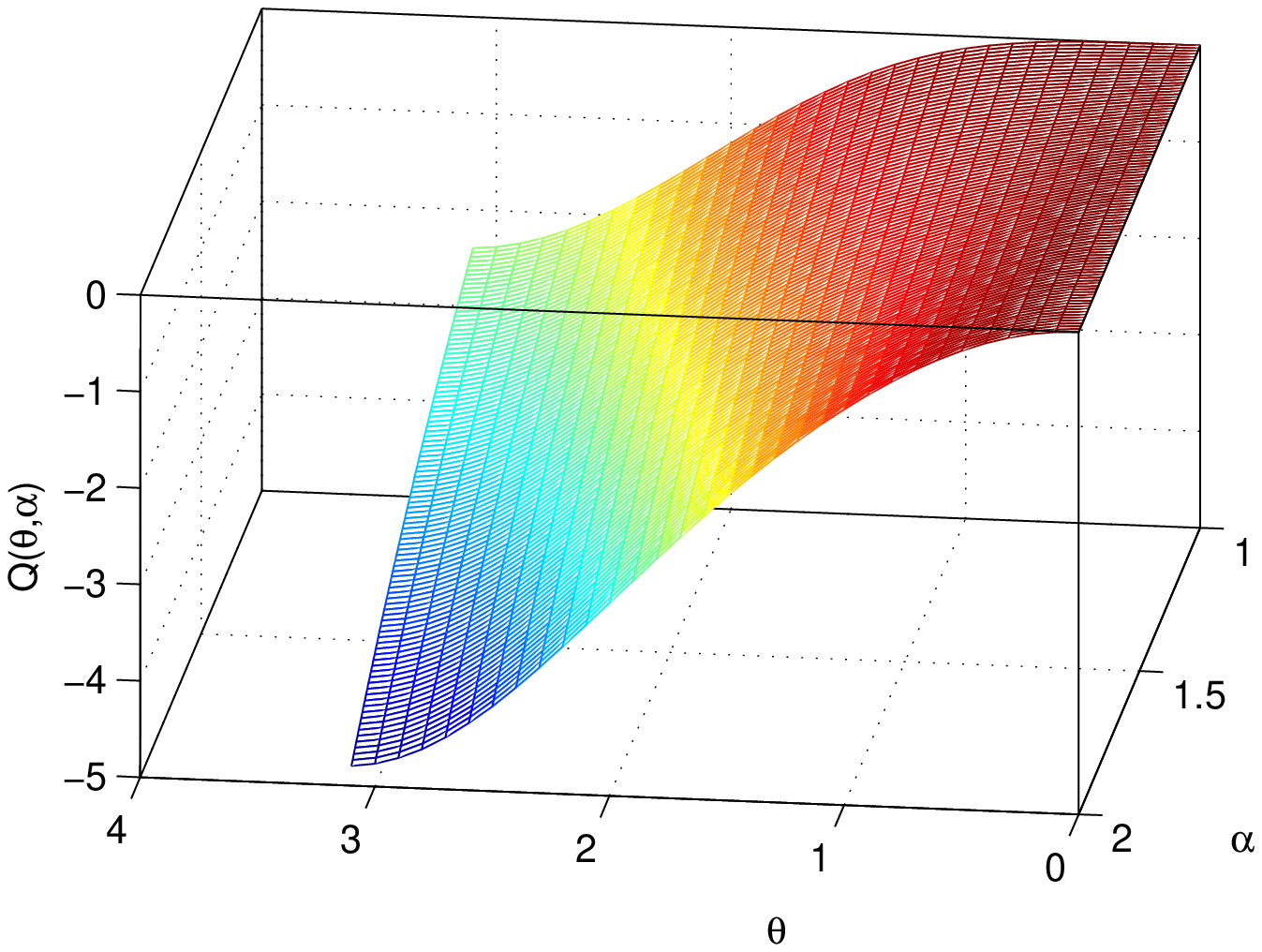}
        \end{minipage}}\\
    \subfigure[$\lambda_{2}=0$]{
        \label{fig:mini:subfig:b} 
        \begin{minipage}[b]{0.35\textwidth}
            \centering
            \includegraphics[scale=.4]{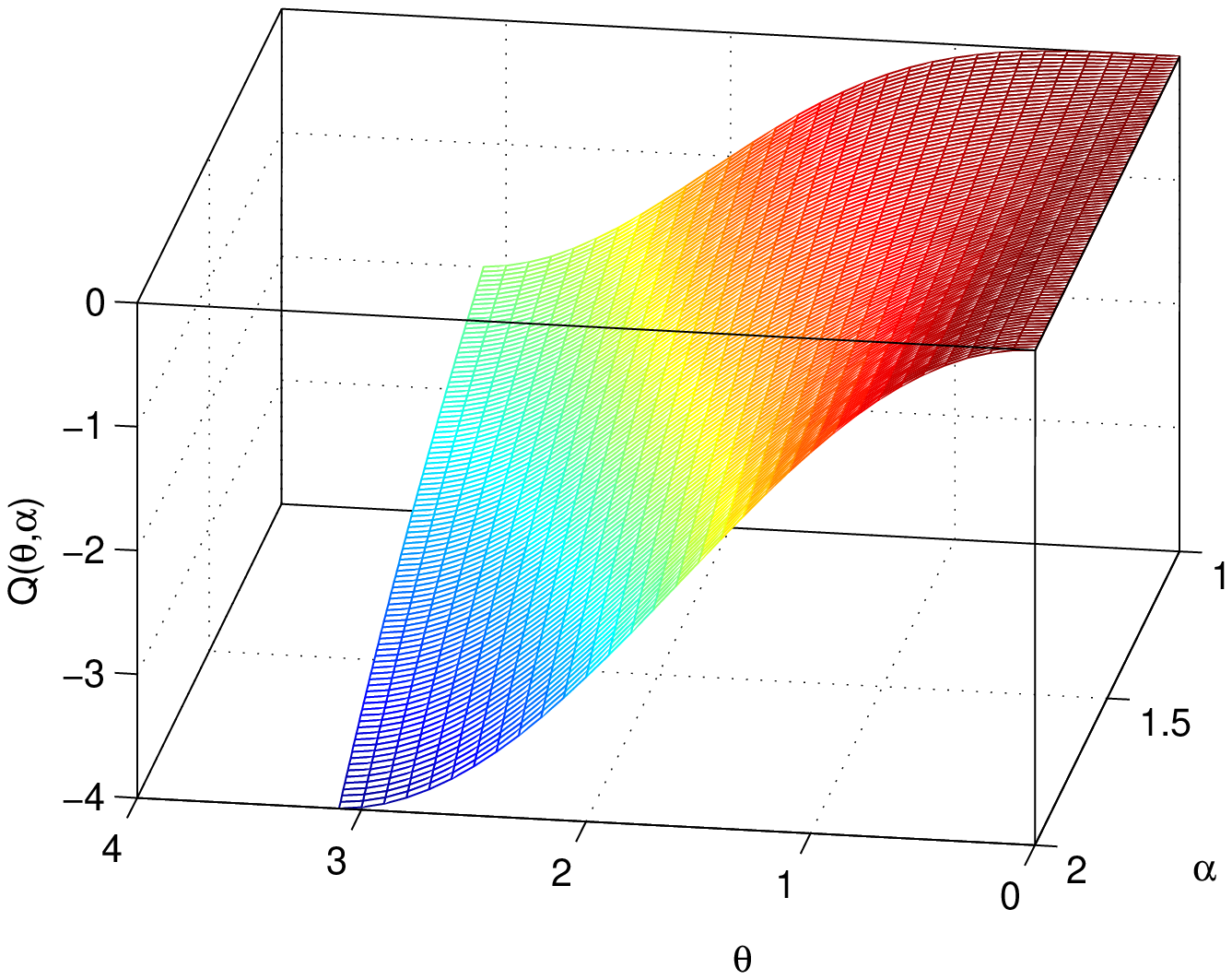}
        \end{minipage}}
    \subfigure[$\lambda_{2}=0.4$]{
        \label{fig:mini:subfig:a} 
        \begin{minipage}[b]{0.35\textwidth}
            \centering
            \includegraphics[scale=.4]{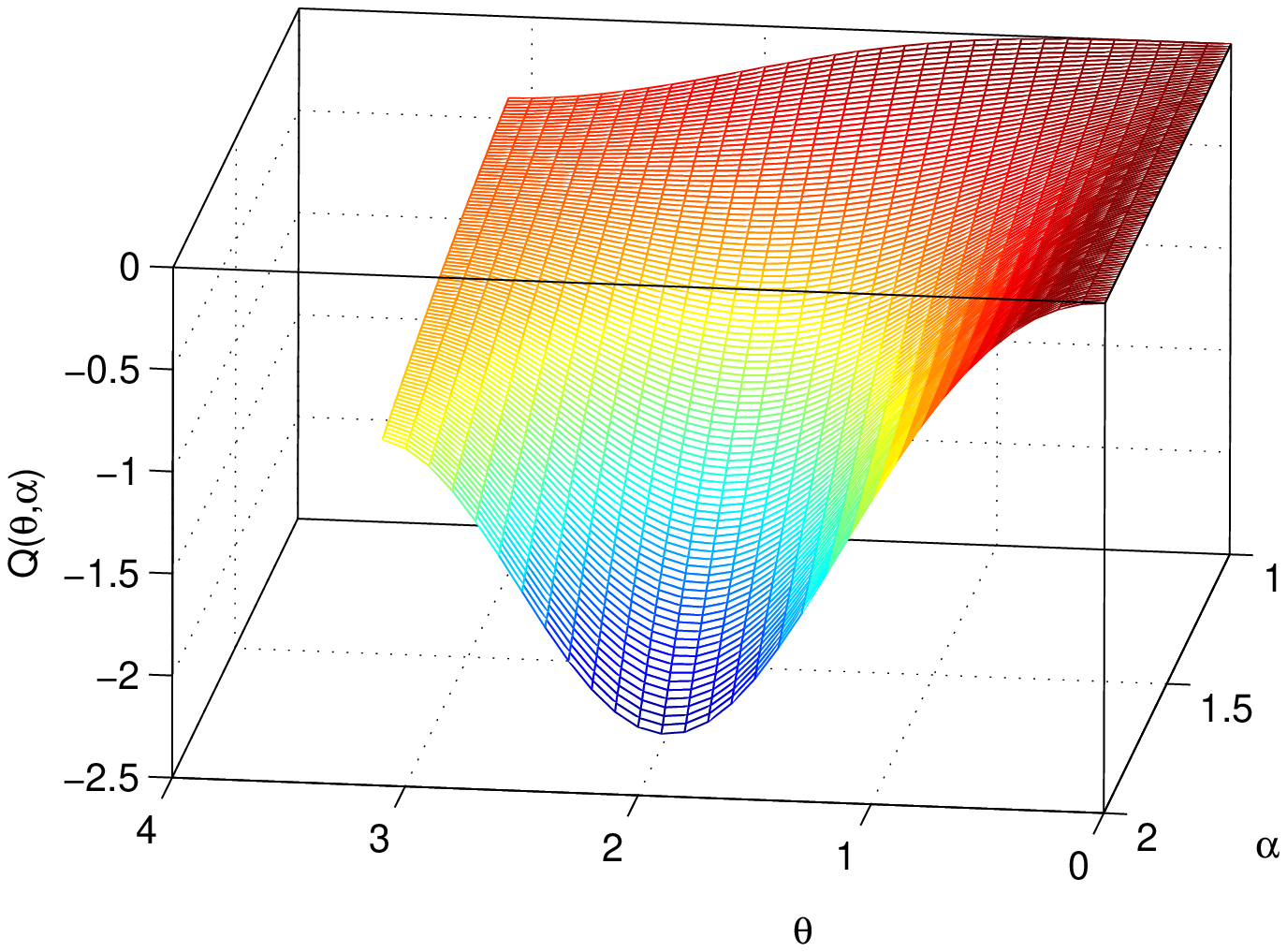}
        \end{minipage}}
    \caption{The trigonometric polynomials $Q(\theta,\alpha)$ in amplification factors $\rho(\theta)$ and the parameters $(\lambda_1,\lambda_2,\lambda_3)$ selected in set
     $\mathcal{S}_{2}=\big\{\frac{2+\alpha}{4}-\frac{\lambda_2}{2},\lambda_2 ~\textrm{is~given},\frac{2-\alpha}{4}-\frac{\lambda_2}{2}\big\}$.}
    \label{fig2} 
\end{figure}
\begin{table}[H]\fontsize{10pt}{12pt}\selectfont
  \begin{center}
  \caption{Numerical errors of Example \ref{exm:1} calculated by the implicit WSGD scheme at $t=1$ for different $\alpha$ with fixed time step $\tau=h^2$ and the parameters $(\lambda_{1},\lambda_{2},\lambda_{3})$ selected in set $\mathcal{S}_{2}$.}\vspace{5pt}
  \begin{tabular*}{\linewidth}{@{\extracolsep{\fill}}*{2}{r}*{10}{c}}
    \toprule
     & & \multicolumn{2}{c}{$\lambda_{2}=-2$} & \multicolumn{2}{c}{$\lambda_{2}=-0.1$}& \multicolumn{2}{c}{$\lambda_{2}=0$}& \multicolumn{2}{c}{$\lambda_{2}=0.4$} \\
    \cline{3-4} \cline{5-6} \cline{7-8}\cline{9-10}\\[-8pt]
    $\alpha$ & $h$ & $\|e\|_{\infty}$ & Rate
    & $\|e\|_{\infty}$ & Rate & $\|e\|_{\infty}$ & Rate& $\|e\|_{\infty}$ & Rate\\
    \toprule
     & 1/10 & 1.17E-03 & -    & 6.20E-04 & -& 5.93E-04 & -      &5.41E-04&-\\
 1.2 & 1/20 & 3.79E-04 & 1.62 & 1.81E-04 & 1.78& 1.71E-04 & 1.79 &1.31E-04&2.05\\
     & 1/40 & 1.39E-04 & 1.45 & 5.24E-05 & 1.79& 4.82E-05 & 1.83 &3.28E-05&2.00\\
    \toprule
     & 1/10 & 3.62E-03 & -    & 7.77E-04 & -& 6.46E-04 & -      &5.11E-04&-\\
1.7  & 1/20 & 9.99E-04 & 1.86 & 2.06E-04 & 1.91& 1.67E-04 & 1.95&6.34E-05&3.01\\
     & 1/40 & 2.41E-04 & 2.05 & 5.16E-05 & 2.00& 4.17E-05 & 2.00&9.73E-06&2.70 \\
    \toprule
     & 1/10 & 4.78E-003 & -    &6.20E-04 & -& 4.52E-04 & - &    4.07E-04&-\\
1.9  & 1/20 & 1.07E-004 & 2.15 & 1.56E-04 & 1.99& 1.10E-04 & 2.04 &7.59E-05&2.42  \\
     & 1/40 & 2.45E-004 & 2.13 & 3.80E-05 & 2.04& 2.69E-05 & 2.03 &1.81E-05&2.07  \\
    \toprule
  \end{tabular*}\label{tab:3}
  \end{center}
\end{table}
\begin{table}[H]\fontsize{10pt}{12pt}\selectfont
  \begin{center}
  \caption{Numerical errors of Example \ref{exm:1} calculated by the Crank-Nicolson-WSGD scheme at $t=1$ for different $\alpha$ with fixed time stepsize $\tau=h$ and the parameters $(\lambda_{1},\lambda_{2},\lambda_{3})$ selected in set $\mathcal{S}_{2}$.}\vspace{5pt}
  \begin{tabular*}{\linewidth}{@{\extracolsep{\fill}}*{2}{r}*{10}{c}}
    \toprule
     & & \multicolumn{2}{c}{$\lambda_{2}=-2$}& \multicolumn{2}{c}{$\lambda_{2}=-0.1$} & \multicolumn{2}{c}{$\lambda_{2}=0$}& \multicolumn{2}{c}{$\lambda_{2}=0.4$} \\
    \cline{3-4} \cline{5-6}\cline{7-8}\cline{9-10} \\[-8pt]
    $\alpha$ & $h$ & $\|e\|_{\infty}$ & Rate
    & $\|e\|_{\infty}$ & Rate & $\|e\|_{\infty}$ & Rate& $\|e\|_{\infty}$ & Rate\\
    \toprule
     & 1/10 & 1.20E-03 & -    & 6.21E-04 & -   & 5.95E-04 & -    &5.52E-04&-\\
 1.2 & 1/20 & 3.77E-04 & 1.67 & 1.79E-04 & 1.80& 1.70E-04 & 1.81 &1.31E-04&2.07\\
     & 1/40 & 1.38E-04 & 1.45 & 5.17E-05 & 1.79& 4.75E-05 & 1.83 &3.33E-05&1.98\\
    \toprule
     & 1/10 & 4.20E-03 & -     & 1.01E-03 & -   & 8.58E-04 & -   &8.09E-04&-\\
1.7  & 1/20 & 1.01E-03 & 2.06  & 2.22E-04 & 2.18& 1.83E-04 & 2.23&9.92E-05&3.03\\
     & 1/40 & 2.41E-04 & 2.07  & 5.13E-05 & 2.12& 4.14E-05 & 2.14&1.17E-06&3.10 \\
    \toprule
     & 1/10 & 6.33E-003 & -    & 9.01E-04 & -   & 6.78E-04 & -   &6.73E-04&-\\
1.9  & 1/20 & 1.27E-003 & 2.31 & 1.96E-04 & 2.20& 1.47E-04 & 2.20&8.17E-05&3.04  \\
     & 1/40 & 2.84E-004 & 2.16 & 4.51E-05 & 2.11& 3.39E-05 & 2.11&1.84E-05&2.15  \\
    \toprule
  \end{tabular*}\label{tab:4}
  \end{center}
\end{table}
\begin{figure}[h]
\centering
    \subfigure[$\lambda_{3}=-0.001$]{
        \label{fig:mini:subfig:a} 
        \begin{minipage}[b]{0.32\textwidth}
            \centering
            \includegraphics[scale=.4]{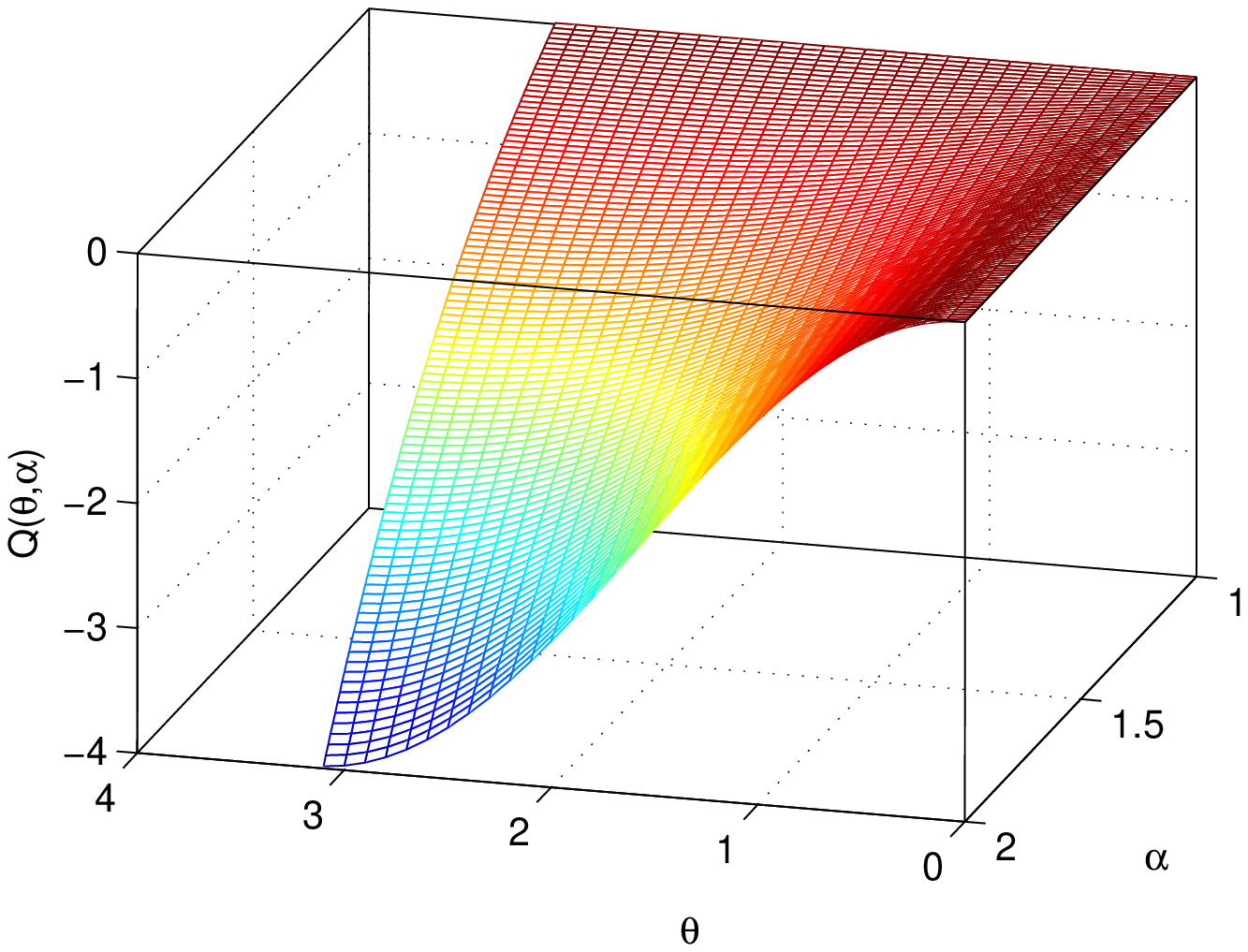}
        \end{minipage}}
    \subfigure[$\lambda_{3}=0$]{
        \label{fig:mini:subfig:a} 
        \begin{minipage}[b]{0.32\textwidth}
            \centering
            \includegraphics[scale=.4]{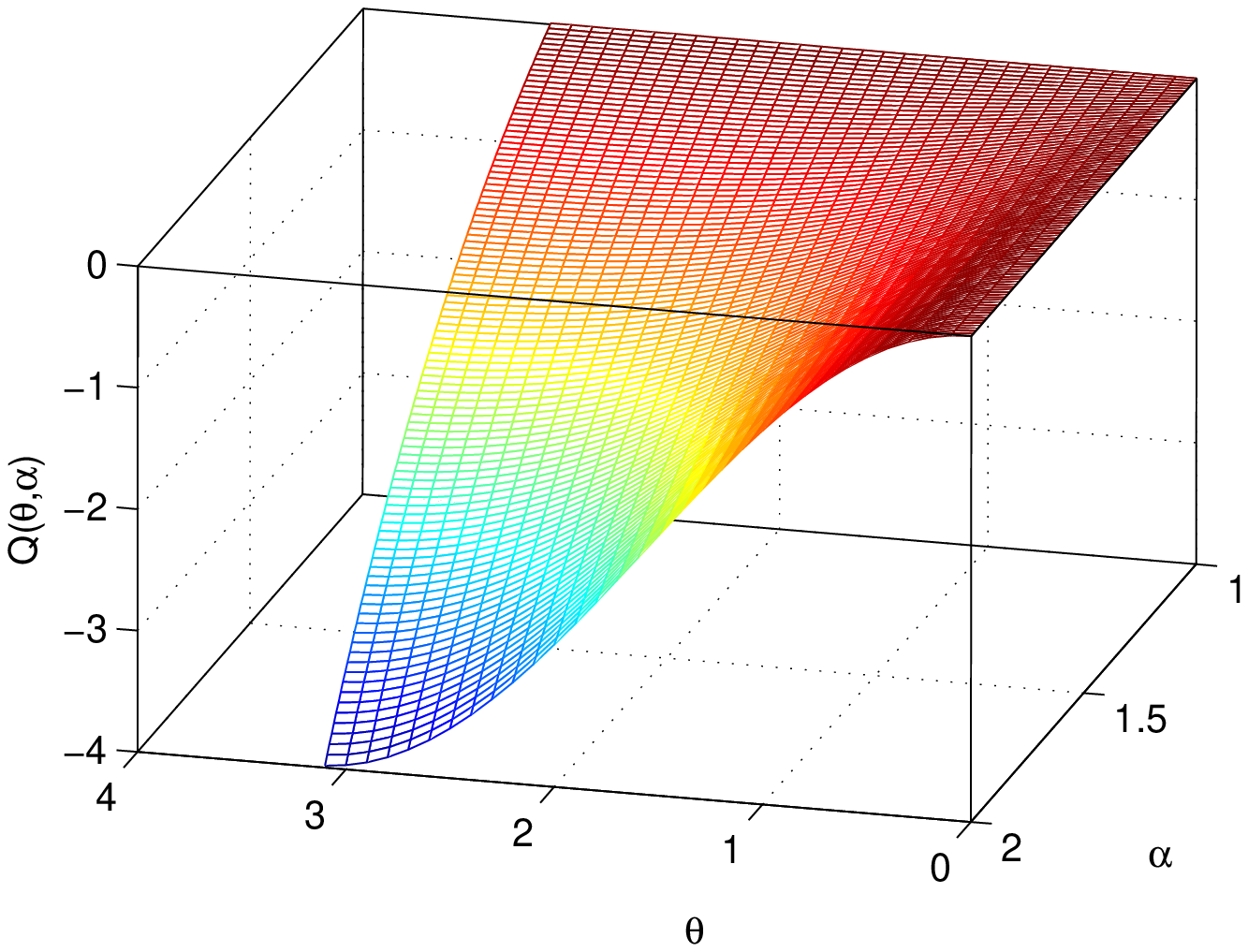}
        \end{minipage}}\\
    \subfigure[$\lambda_{3}=0.001$]{
        \label{fig:mini:subfig:b} 
        \begin{minipage}[b]{0.35\textwidth}
            \centering
            \includegraphics[scale=.4]{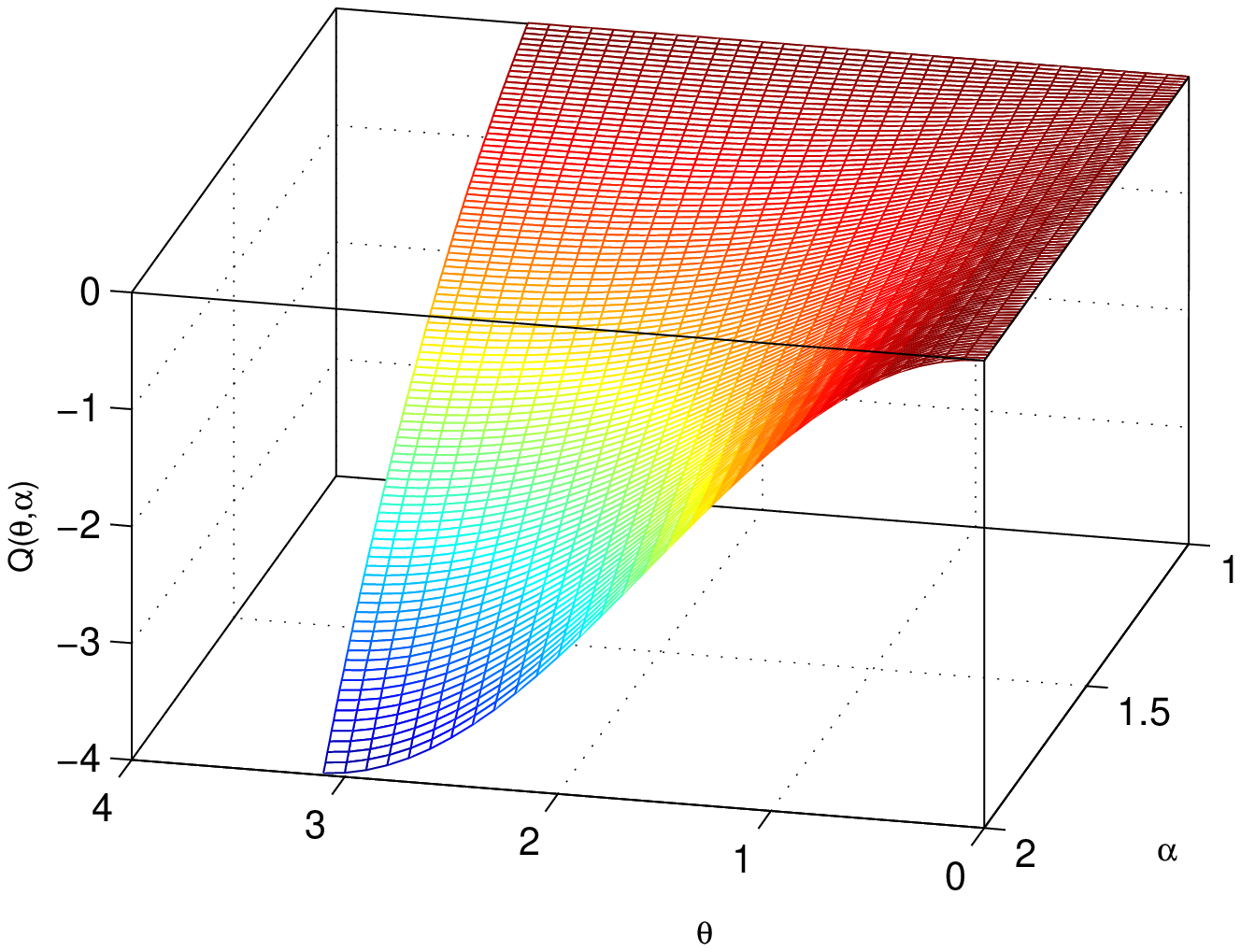}
        \end{minipage}}
    \subfigure[$\lambda_{3}=0.01$]{
        \label{fig:mini:subfig:a} 
        \begin{minipage}[b]{0.35\textwidth}
            \centering
            \includegraphics[scale=.4]{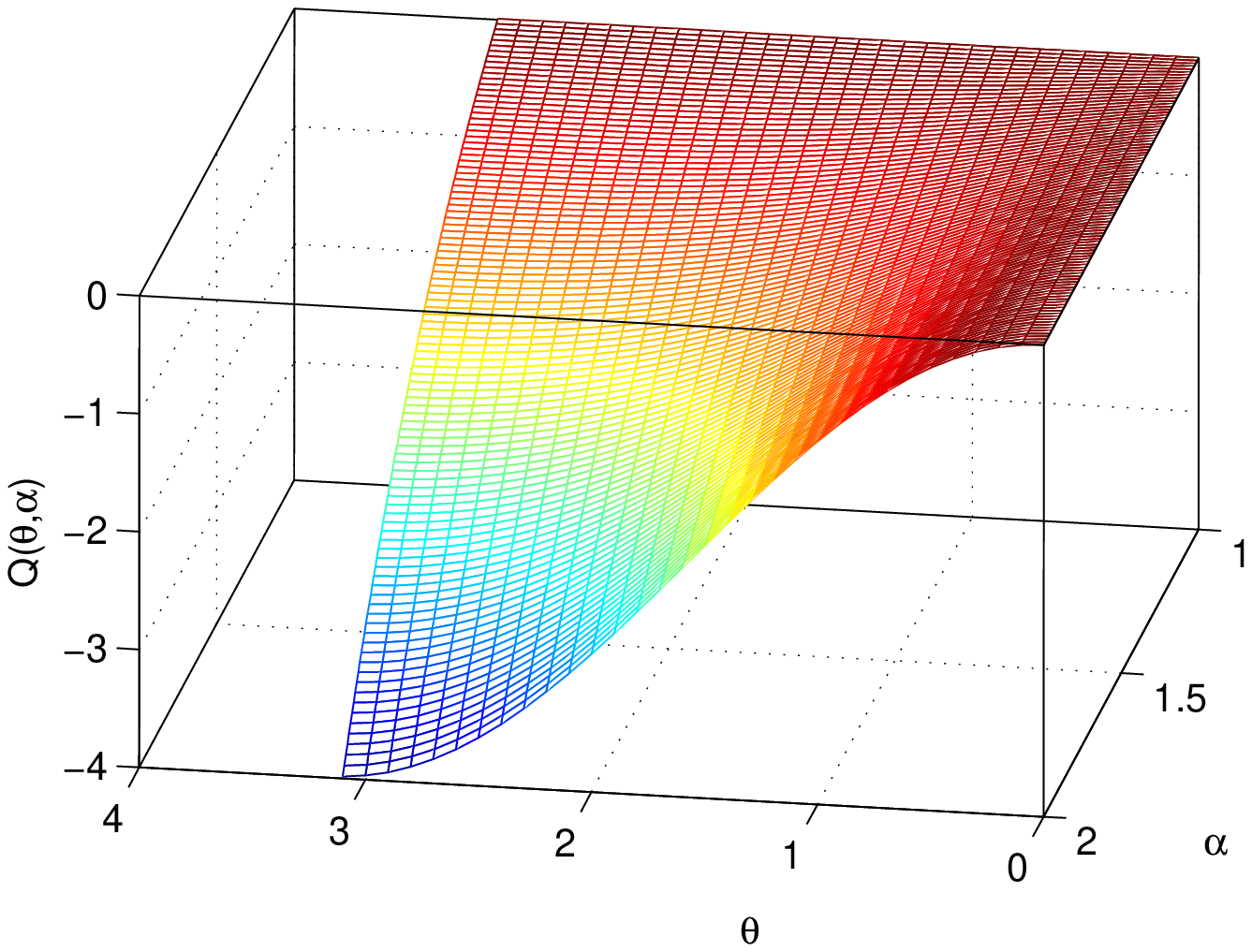}
        \end{minipage}}
    \caption{The trigonometric polynomials $Q(\theta,\alpha)$ in amplification factors $\rho(\theta)$ and the parameters $(\lambda_1,\lambda_2,\lambda_3)$ selected in set
     $\mathcal{S}_{3}=\big\{\frac{\alpha}{2}+\lambda_3,\frac{2-\alpha}{2}-2\lambda_3,\lambda_3 ~\textrm{is~given}\big\}$.}
    \label{fig3} 
\end{figure}
\begin{table}[H]\fontsize{10pt}{12pt}\selectfont
  \begin{center}
  \caption{Numerical errors of Example \ref{exm:1} calculated by the  WSGD scheme at $t=1$ for different $\alpha$ with fixed time stepsize $\tau=h^2$
  and the parameters $(\lambda_{1},\lambda_{2},\lambda_{3})$ selected in set $\mathcal{S}_{3}$.}\vspace{5pt}
  \begin{tabular*}{\linewidth}{@{\extracolsep{\fill}}*{2}{r}*{12}{c}}
    \toprule
     & & \multicolumn{2}{c}{$\lambda_{3}=-0.001$} & \multicolumn{2}{c}{$\lambda_{3}=0$}& \multicolumn{2}{c}{$\lambda_{3}=0.001$}& \multicolumn{2}{c}{$\lambda_{3}=0.01$} \\
    \cline{3-4} \cline{5-6}\cline{7-8} \cline{9-10}\\[-8pt]
    $\alpha$ & $h$ & $\|e\|_{\infty}$ & Rate
    & $\|e\|_{\infty}$ & Rate & $\|e\|_{\infty}$ & Rate& $\|e\|_{\infty}$ & Rate\\
    \toprule
     & 1/10 & 5.93E-04 & -   & 5.93E-04 & -     & 5.93E-04 & -   &5.90E-04&-\\
 1.2 & 1/20 & 1.71E-04 & 1.79& 1.71E-04 & 1.79  & 1.71E-04 & 1.79&1.70E-04&1.79\\
     & 1/40 & 4.83E-05 & 1.83& 4.82E-05 & 1.83  & 4.82E-05 & 1.83&4.78E-05&1.83\\
    \toprule
     & 1/10 & 6.47E-04 & -   & 6.46E-04 & -     & 6.44E-04 & -   &6.33E-04&-\\
1.7  & 1/20 & 1.67E-04 & 1.95& 1.67E-04 & 1.95  & 1.67E-04 & 1.95&1.63E-04&1.96\\
     & 1/40 & 4.18E-05 & 2.00& 4.17E-05 & 2.00  & 4.16E-05 & 2.00&4.07E-05&2.00 \\
    \toprule
     & 1/10 & 4.54E-004 & -   & 4.52E-004 & -   & 4.50E-04 & -   &4.35E-04&-\\
1.9  & 1/20 & 1.10E-004 & 2.04& 1.10E-004 & 2.04& 1.10E-04 & 2.04&1.05E-04&2.05  \\
     & 1/40 & 2.71E-005 & 2.03& 2.70E-005 & 2.03& 2.68E-05 & 2.03&2.59E-05&2.03  \\
    \toprule
  \end{tabular*}\label{tab:5}
  \end{center}
\end{table}

\begin{table}[H]\fontsize{10pt}{12pt}\selectfont
  \begin{center}
  \caption{Numerical errors of Example \ref{exm:1} calculated by the Crank-Nicolson-WSGD scheme
  at $t=1$ for different $\alpha$ with fixed time stepsize $\tau=h$ and the parameters $(\lambda_{1},\lambda_{2},\lambda_{3})$
   selected in set $\mathcal{S}_{3}$.}\vspace{5pt}
  \begin{tabular*}{\linewidth}{@{\extracolsep{\fill}}*{2}{r}*{12}{c}}
   \toprule
     & & \multicolumn{2}{c}{$\lambda_{3}=-0.001$} & \multicolumn{2}{c}{$\lambda_{3}=0$}& \multicolumn{2}{c}{$\lambda_{3}=0.001$}& \multicolumn{2}{c}{$\lambda_{3}=0.01$} \\
    \cline{3-4} \cline{5-6}\cline{7-8}\cline{9-10}\\[-8pt]
    $\alpha$ & $h$ & $\|e\|_{\infty}$ & Rate
    & $\|e\|_{\infty}$ & Rate & $\|e\|_{\infty}$ & Rate\\
    \toprule
     & 1/10 & 5.95E-04 & -     & 5.95E-04 & -   & 5.95E-04 & -   &5.92E-04&-\\
 1.2 & 1/20 & 1.70E-04 & 1.81  & 1.69E-04 & 1.81& 1.69E-04 & 1.81&1.68E-04&1.81\\
     & 1/40 & 4.76E-05 & 1.83  & 4.75E-05 & 1.83& 4.75E-05 & 1.83&4.71E-05&1.84\\
    \toprule
     & 1/10 & 8.60E-04 & -     & 8.59E-04 & -   & 8.57E-04 & -   &8.43E-04&-\\
1.7  & 1/20 & 1.84E-04 & 2.23  & 1.83E-04 & 2.23& 1.83E-04 & 2.23&1.79E-04&2.23\\
     & 1/40 & 4.15E-05 & 2.14  & 4.14E-05 & 2.14& 4.13E-05 & 2.14&4.05E-05&2.14 \\
    \toprule
     & 1/10 & 6.80E-004 & -     & 6.78E-04 & -   & 6.76E-04 & -   &6.56E-04&-\\
1.9  & 1/20 & 1.47E-004 & 2.21  & 1.47E-04 & 2.21& 1.47E-04 & 2.20&1.44E-04&2.19  \\
     & 1/40 & 3.41E-005 & 2.11  & 3.39E-05 & 2.11& 3.38E-05 & 2.12&3.28E-05&2.13 \\
    \toprule
  \end{tabular*}\label{tab:6}
  \end{center}
\end{table}
From Figure \ref{fig2}, we note that with the chosen parameters the absolute values of the trigonometric polynomials $Q$
are almost the same;  it is interesting that the numerical results presented in Tables \ref{tab:5}-\ref{tab:6} are also almost the same.
\begin{table}[H]\fontsize{10pt}{12pt}\selectfont
  \begin{center}
  \caption{The absolute value of trigonometric polynomials $Q(\theta,\alpha)$ with $\theta\in[0,\pi],\alpha\in[1,2]$ and the parameters $(\lambda_{1},\lambda_{2},\lambda_{3})$
   selected in different sets.}\vspace{5pt}
  \begin{tabular*}{\linewidth}{@{\extracolsep{\fill}}*{2}{r}*{12}{c}}
   \toprule
    Sets &  $(\lambda_{1},\lambda_{2},\lambda_{3})$ & $\max_{\theta\in[0,\pi],\alpha\in[1,2]}|Q|$\\
    \toprule
                   & $\lambda_{1}=0.75$ & 2.1122    \\
 $\mathcal{S}_{1}$ & $\lambda_{1}=1.00$ & 4.2422    \\
                   & $\lambda_{1}=1.20$ & 7.1927    \\
                   & $\lambda_{1}=1.50$ & 11.984    \\
    \toprule
                   & $\lambda_{2}=-2.00$ & 19.970    \\
 $\mathcal{S}_{2}$ & $\lambda_{2}=-0.10$ & 4.7969    \\
                   & $\lambda_{2}=0$     & 3.9983    \\
                   & $\lambda_{2}=0.40$  & 2.1937    \\
    \toprule
                   & $\lambda_{3}=-0.001$& 3.9823    \\
 $\mathcal{S}_{3}$ & $\lambda_{3}=0$     & 3.9983    \\
                   & $\lambda_{3}=0.001$ & 4.0142    \\
                   & $\lambda_{3}=0.01$  & 4.1580    \\
    \toprule
  \end{tabular*}\label{tab:60}
  \end{center}
\end{table}
 The values of the trigonometric polynomials $Q(\theta,\alpha)$ in amplification factors $\rho(\theta)$ are shown in
 Figures \ref{fig1}-\ref{fig3}. It is interesting to note that with the decrease of the absolute value of the trigonometric polynomials $Q$ in amplification factors $\rho(\theta)$, the numerical errors are increasing. The reason should be that the numerical error dissipation is weakening with the increase of the value of $Q$.
Hence, in practical computation, we can check the amplification factors $\rho(\theta)$ to get stable and high accurate numerical scheme.
And this results are consist with the classical results \cite{Gustafsson:95,Strikwerda:04}.

\begin{example}\label{exm:2}
The following two dimensional fractional advection diffusion equation
  \begin{equation*}
      u_{t}(x,y,t)+ \partial_{x}u(x,y,t)+ \partial_{y}u(x,y,t)=({_0}D_x^{\alpha}+{_x}D_1^{\alpha})u(x,y,t)
        +({_0}D_y^{\beta}+{_y}D_1^{\beta})u(x,y,t)+f(x,y,t)
  \end{equation*}
  is considered in the domain $\Omega=(0,1)^2$; and with the boundary conditions $u(x,y,t)|_{\partial\Omega}=0$ and the initial condition $u(x,y,0)=x^3(1-x)^3y^3(1-y)^3$; where the source term
\begin{equation*}
     \begin{split}
     f(x,y,t)=-\mathrm{e}^{-t}\Big[& x^3(1-x)^3y^3(1-y)^3-3x^2(1-x)^3y^3(1-y)^3+3x^3(1-x)^2y^3(1-y)^3\\
             &-3x^3(1-x)^3y^2(1-y)^3+3x^3(1-x)^3y^3(1-y)^2\\
             &+\Big(\frac{\Gamma(4)}{\Gamma(4-\alpha)}\big(x^{3-\alpha}+(1-x)^{3-\alpha}\big)
              -\frac{3\Gamma(5)}{\Gamma(5-\alpha)}\big(x^{4-\alpha}+(1-x)^{4-\alpha}\big)\\
             &+\frac{3\Gamma(6)}{\Gamma(6-\alpha)}\big(x^{5-\alpha}+(1-x)^{5-\alpha}\big)
              -\frac{\Gamma(7)}{\Gamma(7-\alpha)}\big(x^{6-\alpha}+(1-x)^{6-\alpha}\big)\Big)y^3(1-y)^3\\
             &+\Big(\frac{\Gamma(4)}{\Gamma(4-\beta)}\big(y^{3-\beta}+(1-y)^{3-\beta}\big)
              -\frac{3\Gamma(5)}{\Gamma(5-\beta)}\big(y^{4-\beta}+(1-y)^{4-\beta}\big)\\
             &+\frac{3\Gamma(6)}{\Gamma(6-\beta)}\big(y^{5-\beta}+(1-y)^{5-\beta}\big)
              -\frac{\Gamma(7)}{\Gamma(7-\beta)}\big(y^{6-\beta}+(1-y)^{6-\beta}\big)\Big)x^3(1-x)^3\Big].
    \end{split}
  \end{equation*}
  The exact solution is given by $u(x,y,t)=\mathrm{e}^{-t}x^3(1-x)^3y^3(1-y)^3$.
\end{example}
In this example we use three numerical schemes:  PR-ADI \eqref{eq:4.16}, Douglas-ADI
\eqref{eq:4.17} and D'yakonov-ADI \eqref{eq:4.18};  the
maximum  errors and their convergence rates to Example \ref{exm:2} approximated
at $t=1$ are listed in Tables \ref{tab:7}-\ref{tab:10}, where $h_x=h_y$; and $(\lambda_{1},\lambda_{2},\lambda_{3})$ and
$(\gamma_{1},\gamma_{2},\gamma_{3})$ denote the weighted parameters of the WSGD operators in $x$ and $y$ directions, respectively.

\begin{table}[H]\fontsize{10pt}{12pt}\selectfont
  \begin{center}
  \caption{The maximum errors and convergent rates to Example \ref{exm:2} approximated at $t=1$
  with $\tau=h_x=h_y$, $\alpha=1.5,\beta=1.8$ and the parameters $(\lambda_{1},\lambda_{2},\lambda_{3})$
   and $(\gamma_{1},\gamma_{2},\gamma_{3})$ selected in set $\mathcal{S}_{1}$.}\vspace{5pt}
  \begin{tabular*}{\linewidth}{@{\extracolsep{\fill}}*{2}{r}*{10}{c}}
    \toprule
     & & \multicolumn{2}{c}{$\lambda_{1}=\gamma_{1}=0.8$} & \multicolumn{2}{c}{$\lambda_{1}=\gamma_{1}=1.0$}& \multicolumn{2}{c}{$\lambda_{1}=\gamma_{1}=1.2$} \\
    \cline{3-4} \cline{5-6}\cline{7-8} \\[-8pt]
    Scheme & $h_x=h_y$ & $\|e\|_{\infty}$ & Rate
    & $\|e\|_{\infty}$ & Rate & $\|e\|_{\infty}$ & Rate\\
    \toprule
             & 1/10 & 4.65E-06 & -     & 8.45E-06 & -    & 1.24E-05 & - \\
    Peaceman-& 1/20 & 1.10E-06 & 2.07  & 2.33E-07 & 1.86 & 3.59E-06 & 1.79  \\
    Rachford-& 1/40 & 2.81E-07 & 1.97  & 5.98E-07 & 1.96 & 9.26E-07 & 1.95   \\
    ADI      & 1/80 & 7.08E-08 & 1.99  & 1.51E-07 & 1.99 & 2.34E-07 & 1.98   \\
    \midrule
              & 1/10 & 3.40E-06 & -     & 7.17E-06 & -    & 1.11E-05 & -  \\
    Douglas-  & 1/20 & 8.05E-07 & 2.08  & 2.02E-07 & 1.83 & 3.27E-06 & 1.76   \\
    ADI       & 1/40 & 2.03E-07 & 1.99  & 5.22E-07 & 1.95 & 8.55E-07 & 1.94   \\
              & 1/80 & 5.17E-08 & 1.98  & 1.33E-07 & 1.99 & 2.16E-07 & 1.98   \\
    \midrule
              & 1/10 & 3.40E-06 & -     & 7.17E-06 & -    & 1.11E-05 & -  \\
    D'yakonov-& 1/20 & 8.05E-07 & 2.08  & 2.02E-07 & 1.83 & 3.27E-06 & 1.76   \\
    ADI       & 1/40 & 2.03E-07 & 1.99  & 5.22E-07 & 1.95 & 8.55E-07 & 1.94   \\
              & 1/80 & 5.17E-08 & 1.98  & 1.33E-07 & 1.99 & 2.16E-07 & 1.98   \\
    \toprule
  \end{tabular*}\label{tab:7}
  \end{center}
\end{table}
\begin{table}[H]\fontsize{10pt}{12pt}\selectfont
  \begin{center}
  \caption{The maximum errors and convergent rates to Example \ref{exm:2} approximated at $t=1$ with $\tau=h_x=h_y$, $\alpha=1.5,\beta=1.8$ and the parameters $(\lambda_{1},\lambda_{2},\lambda_{3})$ and $(\gamma_{1},\gamma_{2},\gamma_{3})$ selected in set $\mathcal{S}_{2}$.}\vspace{5pt}
  \begin{tabular*}{\linewidth}{@{\extracolsep{\fill}}*{2}{r}*{10}{c}}
    \toprule
     & & \multicolumn{2}{c}{$\lambda_{2}=\gamma_{2}=-0.3$} & \multicolumn{2}{c}{$\lambda_{2}=\gamma_{2}=0$}& \multicolumn{2}{c}{$\lambda_{2}=\gamma_{2}=0.3$} \\
    \cline{3-4} \cline{5-6}\cline{7-8} \\[-8pt]
    Scheme & $h_x=h_y$ & $\|e\|_{\infty}$ & Rate
    & $\|e\|_{\infty}$ & Rate & $\|e\|_{\infty}$ & Rate\\
    \toprule
             & 1/10 & 1.00E-05 & -     & 7.13E-06 & -    & 4.28E-06 & - \\
    Peaceman-& 1/20 & 2.80E-06 & 1.84  & 1.87E-06 & 1.93 & 9.82E-07 & 2.13  \\
    Rachford-& 1/40 & 7.17E-07 & 1.97  & 4.75E-07 & 1.98 & 2.42E-07 & 2.02   \\
    ADI      & 1/80 & 1.81E-08 & 1.99  & 1.19E-07 & 1.99 & 6.08E-08 & 1.99   \\
    \midrule
            & 1/10 & 8.75E-06 & -     & 5.86E-06 & -    & 3.03E-06 & - \\
    Douglas-& 1/20 & 2.49E-06 & 1.81  & 1.56E-07 & 1.91 & 6.73E-07 & 2.17   \\
    ADI     & 1/40 & 6.43E-07 & 1.96  & 3.99E-07 & 1.97 & 1.65E-07 & 2.02   \\
            & 1/80 & 1.63E-07 & 1.98  & 1.00E-07 & 1.99 & 4.15E-07 & 1.99   \\
    \midrule
              & 1/10 & 8.75E-06 & -     & 5.86E-06 & -    & 3.03E-06 & -  \\
    D'yakonov-& 1/20 & 2.49E-06 & 1.81  & 1.56E-07 & 1.90 & 6.73E-07 & 2.17   \\
    ADI       & 1/40 & 6.43E-07 & 1.96  & 3.99E-07 & 1.97 & 1.65E-07 & 2.02   \\
              & 1/80 & 1.63E-07 & 1.98  & 1.00E-07 & 1.99 & 4.15E-08 & 1.99   \\
    \toprule
  \end{tabular*}\label{tab:8}
  \end{center}
\end{table}
\begin{table}[H]\fontsize{10pt}{12pt}\selectfont
  \begin{center}
  \caption{The maximum errors and convergent rates to Example \ref{exm:2} approximated at $t=1$
   with $\tau=h_x=h_y$, $\alpha=1.5,\beta=1.8$ and the parameters $(\lambda_{1},\lambda_{2},\lambda_{3})$
    and $(\gamma_{1},\gamma_{2},\gamma_{3})$ selected in set $\mathcal{S}_{3}$.}\vspace{5pt}
  \begin{tabular*}{\linewidth}{@{\extracolsep{\fill}}*{2}{r}*{10}{c}}
    \toprule
     & & \multicolumn{2}{c}{$\lambda_{3}=\gamma_{3}=-0.002$} & \multicolumn{2}{c}{$\lambda_{3}=\gamma_{3}=0$}& \multicolumn{2}{c}{$\lambda_{3}=\gamma_{3}=0.002$} \\
    \cline{3-4} \cline{5-6}\cline{7-8} \\[-8pt]
    Scheme & $h_{x}=h_{y}$ & $\|e\|_{\infty}$ & Rate
    & $\|e\|_{\infty}$ & Rate & $\|e\|_{\infty}$ & Rate\\
    \toprule
             & 1/10 & 5.84E-06 & -     & 5.82E-06 & -    & 5.86E-06 & - \\
    Peaceman-& 1/20 & 1.42E-06 & 2.04  & 1.41E-06 & 2.04 & 1.43E-06 & 2.04  \\
    Rachford-& 1/40 & 3.55E-07 & 2.00  & 3.54E-07 & 2.00 & 3.57E-07 & 2.00   \\
    ADI      & 1/80 & 8.87E-08 & 2.00  & 8.85E-08 & 2.00 & 8.92E-08 & 2.00   \\
    \midrule
            & 1/10 & 4.52E-06 & -     & 4.56E-06 & -    & 4.59E-06 & - \\
    Douglas-& 1/20 & 1.10E-06 & 2.04  & 1.11E-06 & 2.04 & 1.12E-06 & 2.03   \\
    ADI     & 1/40 & 2.73E-07 & 2.00  & 2.76E-07 & 2.00 & 2.79E-07 & 2.00   \\
            & 1/80 & 6.84E-08 & 2.00  & 6.91E-08 & 2.00 & 6.99E-08 & 2.00   \\
    \midrule
              & 1/10 & 4.52E-06 & -     & 4.52E-06 & -    & 4.59E-06 & -  \\
    D'yakonov-& 1/20 & 1.09E-06 & 2.04  & 1.11E-06 & 2.03 & 1.12E-06 & 2.03   \\
    ADI       & 1/40 & 2.73E-07 & 2.00  & 2.76E-07 & 2.00 & 2.79E-07 & 2.00   \\
              & 1/80 & 6.84E-08 & 2.00  & 6.91E-08 & 2.00 & 6.99E-07 & 2.00   \\
    \toprule
  \end{tabular*}\label{tab:9}
  \end{center}
\end{table}

\begin{table}[H]\fontsize{10pt}{12pt}\selectfont
  \begin{center}
  \caption{The maximum errors and convergent rates to Example \ref{exm:2}
  approximated at $t=1$ with $\tau=h_x=h_y$, $\alpha=1.5,\beta=1.8$ and the parameters
  $(\lambda_{1},\lambda_{2},\lambda_{3})$ selected in set $\mathcal{S}_{3}$ and
  $(\gamma_{1},\gamma_{2},\gamma_{3})$ selected in set $\mathcal{S}_{2}$.}\vspace{5pt}
  \begin{tabular*}{\linewidth}{@{\extracolsep{\fill}}*{2}{r}*{8}{c}}
    \toprule
     & & \multicolumn{2}{c}{$\lambda_{3}=0.02,~\gamma_{2}=0.4$} & \multicolumn{2}{c}{$\lambda_{3}=0.02,~\gamma_{2}=0.5$} \\
    \cline{3-4} \cline{5-6} \\[-8pt]
    Scheme & $h_x=h_y$ & $\|e\|_{\infty}$ & Rate
    & $\|e\|_{\infty}$ & Rate \\
    \toprule
             & 1/10 & 3.72E-06 & -     & 3.96E-06 & -  \\
    Peaceman-& 1/20 & 8.78E-07 & 2.08  & 8.96E-07 & 2.17 \\
    Rachford-& 1/40 & 2.19E-07 & 2.00  & 2.18E-07 & 2.04 \\
    ADI      & 1/80 & 5.55E-08 & 1.98  & 5.37E-08 & 2.02 \\
    \midrule
              & 1/10 & 2.55E-06 & -     & 3.11E-06 & -  \\
    Douglas-& 1/20 & 6.20E-07 & 2.04  & 7.60E-07 & 2.03 \\
    ADI       & 1/40 & 1.61E-07 & 1.95  & 1.81E-07 & 2.07 \\
              & 1/80 & 4.05E-08 & 1.99  & 4.45E-08 & 2.02 \\
    \midrule
              & 1/10 & 2.55E-06 & -     & 3.11E-06 & -  \\
    D'yakonov-& 1/20 & 6.20E-07 & 2.04  & 7.60E-07 & 2.03 \\
    ADI       & 1/40 & 1.61E-07 & 1.95  & 1.81E-07 & 2.07 \\
              & 1/80 & 4.05E-08 & 1.99  & 4.45E-08 & 2.02 \\
    \toprule
  \end{tabular*}\label{tab:10}
  \end{center}
\end{table}

From Tables \ref{tab:7}-\ref{tab:9}, it can seen that the numerical errors of the Douglas-ADI
\eqref{eq:4.17} and D'yakonov-ADI \eqref{eq:4.18} schemes are identical,
and both are smaller than those of Peaceman-Rachford-ADI \eqref{eq:4.16} scheme.
By examining the values of Tables \ref{tab:7}-\ref{tab:9}, we immediately notice that the numerical errors are
increasing with the decrease of the absolute values of the trigonometric polynomials $Q(\theta,\alpha)$. This is because of the different degree of the numerical error dissipation. 
This phenomenon also happen in the one dimensional case being reported in Tables \ref{tab:1}-\ref{tab:8}.
In Table \ref{tab:10}, we list the numerical errors and convergent rates for three ADI schemes.  The weighted parameters  $(\lambda_{1},\lambda_{2},\lambda_{3})$ and $(\gamma_{1},\gamma_{2},\gamma_{3})$ in $x$ and $y$ directions are selected in
different sets.
\begin{table}[H]\fontsize{10pt}{12pt}\selectfont
  \begin{center}
  \caption{The absolute value of trigonometric polynomials $Q(\theta_{x},\alpha)$ and $Q(\theta_{x},\beta)$ with the parameters $(\lambda_{1},\lambda_{2},\lambda_{3})$ and $(\gamma_{1},\gamma_{2},\gamma_{3})$
   selected in different sets.}\vspace{5pt}
  \begin{tabular*}{\linewidth}{@{\extracolsep{\fill}}*{2}{r}*{12}{c}}
   \toprule
    Sets &  $(\lambda_{1},\lambda_{2},\lambda_{3})$ and $(\gamma_{1},\gamma_{2},\gamma_{3})$ &
     $\max_{\theta_{x},\theta_{y}\in[0,\pi],\alpha,\beta\in[1,2]}\big\{|Q(\theta_{x},\alpha)|,|Q(\theta_{x},\alpha)|\big\}$\\
    \toprule
                   & $\lambda_{1}=\gamma_{1}=0.80$ & 2.3979    \\
 $\mathcal{S}_{1}$ & $\lambda_{1}=\gamma_{1}=1.00$ & 4.2422    \\
                   & $\lambda_{1}=\gamma_{1}=1.20$ & 7.1927    \\
    \toprule
                   & $\lambda_{2}=\gamma_{2}=-0.3$ & 6.3941    \\
 $\mathcal{S}_{2}$ & $\lambda_{2}=\gamma_{2}=0$    & 3.9983    \\
                   & $\lambda_{2}=\gamma_{2}=0.3$  & 2.3316    \\
    \toprule
                   & $\lambda_{3}=\gamma_{3}=-0.002$& 3.9663    \\
 $\mathcal{S}_{3}$ & $\lambda_{3}=\gamma_{3}=0$     & 3.9983    \\
                   & $\lambda_{3}=\gamma_{3}=0.002$ & 4.0302    \\
                   \toprule
                   & $\lambda_{3}=0.002,\gamma_{2}=0.4$& 2.1937    \\
 $\mathcal{S}_{3},\mathcal{S}_{2}$ & $\lambda_{3}=0.002,\gamma_{2}=0.5$     & 2.1122    \\
    \toprule
  \end{tabular*}\label{tab:11}
  \end{center}
\end{table}
\section{Concluding remarks}\label{sec:6}
With the firstly introducing of the second order WSGD operators for space discretizations in \cite{Tian:11}, this paper further develops a new family of second order WSGD operators, called second order WSGD operators II. The second order WSGD operators II are used to discretize the one and two dimensional space fractional advection diffusion equations. The sufficient stability conditions are obtained, and by numerical test it is easy to get the effective region of the parameters. The extensive numerical experiments are performed to show the powerfulness of the second order WSGD operators II for space discretizations. In particular, by enhancing the numerical error dissipation, i.e., by adjusting the parameters to decrease the amplification factor, the accuracy can be imporved.


\section*{Acknowledgements}
The authors thank Prof Yujiang Wu for his constant encouragement and support.
This research was partially supported by the National Natural Science Foundation of China under Grant  No. 11271173,
 the Starting Research Fund from the Xi'an University of
Technology under Grant No. 108-211206 and the Scientific
Research Program Funded by Shaanxi Provincial Education Department under Grant No. 2013JK0581.
Can Li would like to thank Dr WenYi Tian for the helpful discussions during preparing this work.


  \end{document}